\documentclass[11pt]{article}

\usepackage[dvipsnames]{xcolor}

\usepackage{epsfig,amsmath,amsfonts,amssymb,latexsym,color,amsthm,hhline,multirow,subfigure,graphicx,bm}
\usepackage{enumerate}

\usepackage{calrsfs} % secondary calligraphic font
\DeclareMathAlphabet{\pazocal}{OMS}{zplm}{m}{n}

\usepackage{tikz}
\usetikzlibrary{arrows, decorations.pathmorphing}
\usepackage{pgfplots}
\usetikzlibrary{hobby}

\theoremstyle{plain}
\newtheorem{theorem}{Theorem}%[section]
\theoremstyle{definition}

\theoremstyle{definition}

\theoremstyle{definition}

\theoremstyle{definition}

\theoremstyle{plain}

\theoremstyle{plain}
\newtheorem{proposition}[theorem]{Proposition}
\theoremstyle{plain}
\newtheorem{lemma}[theorem]{Lemma}
\theoremstyle{plain}

\theoremstyle{definition}
\newtheorem{remark}[theorem]{Remark}
\theoremstyle{plain}

\newcommand{\E}{{\mathbb E}}
\newcommand{\R}{{\mathbb R}}

\newcommand{\Z}{{\mathbb Z}}
\newcommand{\PP}{{\mathbb P}}

\newcommand{\vep}{\varepsilon}

\def\geo{\textup{Geo}}

\begin{document}

%%%%%%%%%%%%%%%%%%%%%%%% TITLE PAGE %%%%%%%%%%%%%%%%%%%%%%%%

\title{Surviving from the tip of a cone in\\ competing first-passage percolation}

\author{\renewcommand{\thefootnote}{\arabic{footnote}}
\renewcommand{\thefootnote}{\arabic{footnote}}
Daniel Ahlberg%
\footnotemark[1]\hspace{0.5cm}
%\\
\renewcommand{\thefootnote}{\arabic{footnote}}
Maria Deijfen%
\footnotemark[1]\hspace{0.5cm}
%\\
\renewcommand{\thefootnote}{\arabic{footnote}}
Matteo Sfragara\,%
\footnotemark[2]
}

\footnotetext[1]{%
Department of Mathematics, Stockholm University, Sweden}
\footnotetext[2] {Department of Mathematics, University of Padova, Italy}

\date{July 2026}
\maketitle

\begin{abstract}
\noindent In two-type first passage percolation on $\mathbb{Z}^2$, two entities compete to capture the sites of the lattice. The entities spread between nearest neighbor sites at times specified by random passage times associated with the edges. We consider the case when both types have the same passage time distribution, with one type starting at the origin and the other from an infinite cone with tip at the origin and pointing in direction $\theta$. Itai Benjamini has suggested that the type starting at the origin can grow unboundedly if and only if the slope of the cone is strictly smaller than $\pi/2$, so that the cone does not fill a whole half-plane. The main result is that this is correct for any $\theta$ such that the asymptotic shape of the one-type process has a tangent line with direction $\theta$. The proofs are based on a description of infinite time-minimizing paths in terms of Busemann functions together with local modification arguments.

\vspace{0.3cm}


\vspace{0.2cm}

\end{abstract}

\section{Introduction}
\label{sec:intro}

The two-type Richardson model is a model for competition between two growing entities, referred to as type~1 and type~2, on the integer lattice $\Z^d$, introduced by H\"aggstr\"om and Pemantle~\cite{HP98}. In the model, an uninfected site becomes type $i$ infected at rate $\lambda_i$ times the number of type $i$ infected neighbors and, once infected, a site remains infected by the given type forever. Let $G_i$ denote the event that sites arbitrarily far from the origin are eventually type $i$ infected. When both types start from bounded initial sets, it is clear that each of $G_1$ and $G_2$ fails with positive probability, regardless of the values of $\lambda_1$ and $\lambda_2$, since these events can then be achieved in that one of the types is enclosed by the other at some finite time. The central question is whether the event $G=G_1\cap G_2$, that is that both types grow unboundedly, has positive probability. 

Note that by rescaling of time, it will suffice to consider the case when $\lambda_1 = 1$ and $\lambda_2 = \lambda$, for $\lambda>0$. In \cite{DH06b}, Deijfen and H\"aggstr\"om showed that the answer to the question of coexistence does not depend on the precise choice of the initial sets, as long as these are both bounded. It was conjectured by H\"aggstr\"om and Pemantle~\cite{HP98} that $G$ occurs with positive probability if and only if $\lambda=1$. The ``if''-direction was proved in~\cite{HP98} for $d=2$, and independently by Garet and Marchand~\cite{GM05} and Hoffman~\cite{H05} for $d\ge2$. The ``only if''-direction remains open, but partial and related results have been obtained in \cite{ADH20,DH06a,GM08,HP00}.

Deijfen and H\"aggstr\"om~\cite{DH07} studied the Richardson model in a setting where one of the types starts from an infinite set. If both types start from infinite sets, the event $G$ trivially occurs, so for the event to be interesting one of the two types must start from a finite set. The question then is if the type starting from a finite set can grow unboundedly. In order to describe the results in~\cite{DH07}, we express a point $x \in \Z^d$ by $x=(x_1, \dots, x_d)$, and let
$$
\pazocal{H} := \{x\in\Z^d : x_2 \le 0\}\quad\text{and}\quad\pazocal{L} := \{x\in\Z^d: x_2 \leq 0, x_i = 0\,\, \forall  i \neq 2\}.
$$
Write $I(\pazocal{H})$ (resp.\ $I(\pazocal{L})$) for the initial configuration where all points in $\pazocal{H} \setminus \{0\}$ (resp.\ $\pazocal{L} \setminus \{0\}$) are type 1 infected and 0 is type 2 infected; see Figure~\ref{fig:HandL}. It was shown in \cite{DH07} that, when starting from $I(\pazocal{H})$, the event $G_2$ occurs with positive probability if and only if $\lambda >1$, while, when starting from $I(\pazocal{L})$, the event $G_2$ has positive probability if and only if $\lambda \geq 1$. 

\begin{figure}[htbp]
\begin{center}
\begin{minipage}{.4\textwidth}
\begin{tikzpicture}[scale=.5]
% axes
\draw (-4,0) -- (4,0); % x axis
\draw (0,-4) -- (0,4); % y axis
% dashed lines 
\draw[draw = blue, dashed] (-4.9,0) -- (-4.1,0);
\draw[draw = blue, dashed] (4.3,0) -- (4.9,0);
\draw[dashed] (0, 4.3) -- (0, 4.9);
\draw[dashed] (0, -4.3) -- (0, -4.9);
% initial sites
\textcolor{red}{\draw (0,0) node {$\bullet$};}
\textcolor{blue}{\draw (0.8,0) node {$\bullet$};}
\textcolor{blue}{\draw (1.6,0) node {$\bullet$};}
\textcolor{blue}{\draw (2.4,0) node {$\bullet$};}
\textcolor{blue}{\draw (3.2,0) node {$\bullet$};}
\textcolor{blue}{\draw (4,0) node {$\bullet$};}
\textcolor{blue}{\draw (-0.8,0) node {$\bullet$};}
\textcolor{blue}{\draw (-1.6,0) node {$\bullet$};}
\textcolor{blue}{\draw (-2.4,0) node {$\bullet$};}
\textcolor{blue}{\draw (-3.2,0) node {$\bullet$};}
\textcolor{blue}{\draw (-4,0) node {$\bullet$};}
\textcolor{blue}{\draw (0,-0.8) node {$\bullet$};}
\textcolor{blue}{\draw (0.8,-0.8) node {$\bullet$};}
\textcolor{blue}{\draw (1.6,-0.8) node {$\bullet$};}
\textcolor{blue}{\draw (2.4,-0.8) node {$\bullet$};}
\textcolor{blue}{\draw (3.2,-0.8) node {$\bullet$};}
\textcolor{blue}{\draw (4,-0.8) node {$\bullet$};}
\textcolor{blue}{\draw (-0.8,-0.8) node {$\bullet$};}
\textcolor{blue}{\draw (-1.6,-0.8) node {$\bullet$};}
\textcolor{blue}{\draw (-2.4,-0.8) node {$\bullet$};}
\textcolor{blue}{\draw (-3.2,-0.8) node {$\bullet$};}
\textcolor{blue}{\draw (-4,-0.8) node {$\bullet$};}
\textcolor{blue}{\draw (0,-1.6) node {$\bullet$};}
\textcolor{blue}{\draw (0.8,-1.6) node {$\bullet$};}
\textcolor{blue}{\draw (1.6,-1.6) node {$\bullet$};}
\textcolor{blue}{\draw (2.4,-1.6) node {$\bullet$};}
\textcolor{blue}{\draw (3.2,-1.6) node {$\bullet$};}
\textcolor{blue}{\draw (4,-1.6) node {$\bullet$};}
\textcolor{blue}{\draw (-0.8,-1.6) node {$\bullet$};}
\textcolor{blue}{\draw (-1.6,-1.6) node {$\bullet$};}
\textcolor{blue}{\draw (-2.4,-1.6) node {$\bullet$};}
\textcolor{blue}{\draw (-3.2,-1.6) node {$\bullet$};}
\textcolor{blue}{\draw (-4,-1.6) node {$\bullet$};}
\textcolor{blue}{\draw (0,-2.4) node {$\bullet$};}
\textcolor{blue}{\draw (0.8,-2.4) node {$\bullet$};}
\textcolor{blue}{\draw (1.6,-2.4) node {$\bullet$};}
\textcolor{blue}{\draw (2.4,-2.4) node {$\bullet$};}
\textcolor{blue}{\draw (3.2,-2.4) node {$\bullet$};}
\textcolor{blue}{\draw (4,-2.4) node {$\bullet$};}
\textcolor{blue}{\draw (-0.8,-2.4) node {$\bullet$};}
\textcolor{blue}{\draw (-1.6,-2.4) node {$\bullet$};}
\textcolor{blue}{\draw (-2.4,-2.4) node {$\bullet$};}
\textcolor{blue}{\draw (-3.2,-2.4) node {$\bullet$};}
\textcolor{blue}{\draw (-4,-2.4) node {$\bullet$};}
\textcolor{blue}{\draw (0,-3.2) node {$\bullet$};}
\textcolor{blue}{\draw (0.8,-3.2) node {$\bullet$};}
\textcolor{blue}{\draw (1.6,-3.2) node {$\bullet$};}
\textcolor{blue}{\draw (2.4,-3.2) node {$\bullet$};}
\textcolor{blue}{\draw (3.2,-3.2) node {$\bullet$};}
\textcolor{blue}{\draw (4,-3.2) node {$\bullet$};}
\textcolor{blue}{\draw (-0.8,-3.2) node {$\bullet$};}
\textcolor{blue}{\draw (-1.6,-3.2) node {$\bullet$};}
\textcolor{blue}{\draw (-2.4,-3.2) node {$\bullet$};}
\textcolor{blue}{\draw (-3.2,-3.2) node {$\bullet$};}
\textcolor{blue}{\draw (-4,-3.2) node {$\bullet$};}
\textcolor{blue}{\draw (0,-4) node {$\bullet$};}
\textcolor{blue}{\draw (0.8,-4) node {$\bullet$};}
\textcolor{blue}{\draw (1.6,-4) node {$\bullet$};}
\textcolor{blue}{\draw (2.4,-4) node {$\bullet$};}
\textcolor{blue}{\draw (3.2,-4) node {$\bullet$};}
\textcolor{blue}{\draw (4,-4) node {$\bullet$};}
\textcolor{blue}{\draw (-0.8,-4) node {$\bullet$};}
\textcolor{blue}{\draw (-1.6,-4) node {$\bullet$};}
\textcolor{blue}{\draw (-2.4,-4) node {$\bullet$};}
\textcolor{blue}{\draw (-3.2,-4) node {$\bullet$};}
\textcolor{blue}{\draw (-4,-4) node {$\bullet$};}
\end{tikzpicture}
\end{minipage}
\begin{minipage}{.4\textwidth}
\begin{tikzpicture}[scale=.5]
% axes
\draw (-4,0) -- (4,0); % x axis
\draw (0,-4) -- (0,4); % y axis
% dashed lines
\draw[dashed] (-4.9,0) -- (-4.1,0);
\draw[dashed] (4.3,0) -- (4.9,0);
\draw[dashed] (0, 4.3) -- (0, 4.9);
\draw[draw = blue, dashed] (0, -4.3) -- (0, -4.9);
% sites
\textcolor{red}{\draw (0,0) node {$\bullet$};}
\textcolor{blue}{\draw (0, -0.8) node {$\bullet$};}
\textcolor{blue}{\draw (0, -1.6) node {$\bullet$};}
\textcolor{blue}{\draw (0, -2.4) node {$\bullet$};}
\textcolor{blue}{\draw (0, -3.2) node {$\bullet$};}
\textcolor{blue}{\draw (0, -4) node {$\bullet$};}
%\draw[draw=blue,thick] (0,-3) -- (4.5,1.5) -- (14,1.5);
\end{tikzpicture}
\end{minipage}
\end{center}
\vspace{-0.4cm}
\caption{\small Initial configurations $I(\pazocal{H})$ and $I(\pazocal{L})$ on $\Z^2$ where type 1 and type 2 sites are coloured blue and red, respectively.}
\label{fig:HandL}
\end{figure}
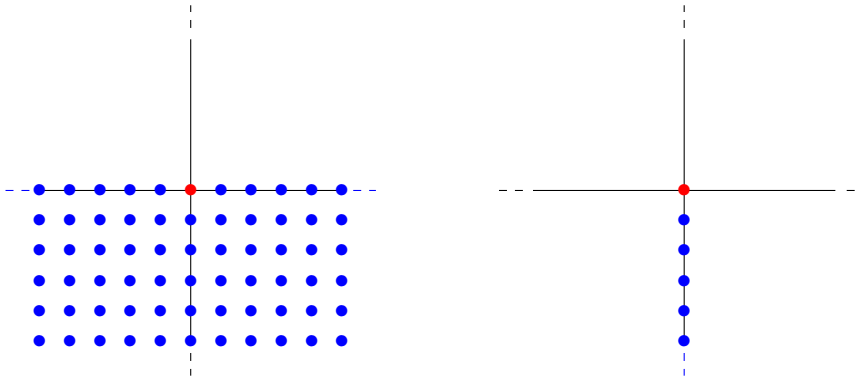

We shall henceforth restrict attention to the planar case $d=2$ with equal growth rates $\lambda=1$. By the results in~\cite{DH07}, a type 2 infection starting at the origin cannot survive in case type 1 occupies the remainder of the horizontal axis (or equivalently, the lower half-plane), but it can survive if type 1 occupies only the vertical half-axis; for instance by rushing off in the opposite direction. This leads to the following question, formulated in~\cite{DH07} and attributed to Itai Benjamini: Let $\pazocal{C}_{\alpha}$ denote the cone with apex at the origin and angle $\alpha$ around the negative $x_2$-axis (so that the total angle of the cone is $2\alpha$). Consider the initial configuration $I(\pazocal{C}_{\alpha})$ with all points in $\pazocal{C}_{\alpha} \setminus \{0\}$ type 1 infected and the origin type 2 infected; see Figure~\ref{fig:cone}. Benjamini asked what is the maximal angle $\alpha$ such that type 2 survives with positive probability when started from the initial condition $I(\pazocal{C}_{\alpha})$?

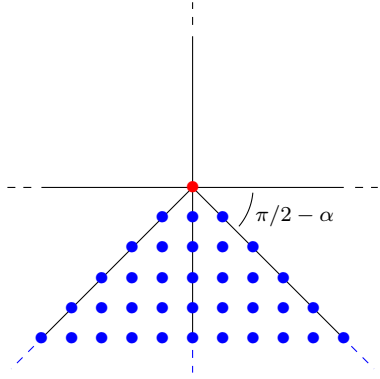
\begin{figure}[htbp]
\begin{center}
\begin{tikzpicture}[scale=.5]
% axes
\draw (-4,0) -- (4,0); % x axis
\draw (0,-4) -- (0,4); % y axis
% dashed lines
\draw[dashed] (-4.9,0) -- (-4.1,0);
\draw[dashed] (4.3,0) -- (4.9,0);
\draw[dashed] (0, 4.3) -- (0, 4.9);
\draw[draw = blue, dashed] (0, -4.3) -- (0, -4.9);
% cone
\draw (-4,-4) -- (0,0) -- (4,-4);
\draw[draw = blue, dashed] (-4.2,-4.2) -- (-4.8,-4.8);
\draw[draw = blue, dashed] (4.2,-4.2) -- (4.8,-4.8);
% sites
\textcolor{red}{\draw (0,0) node {$\bullet$};}
\textcolor{blue}{\draw (0, -0.8) node {$\bullet$};}
\textcolor{blue}{\draw (0, -1.6) node {$\bullet$};}
\textcolor{blue}{\draw (0, -2.4) node {$\bullet$};}
\textcolor{blue}{\draw (0, -3.2) node {$\bullet$};}
\textcolor{blue}{\draw (0, -4) node {$\bullet$};}
\textcolor{blue}{\draw (-0.8, -0.8) node {$\bullet$};}
\textcolor{blue}{\draw (-1.6, -1.6) node {$\bullet$};}
\textcolor{blue}{\draw (-2.4, -2.4) node {$\bullet$};}
\textcolor{blue}{\draw (-3.2, -3.2) node {$\bullet$};}
\textcolor{blue}{\draw (-4, -4) node {$\bullet$};}
\textcolor{blue}{\draw (0.8, -0.8) node {$\bullet$};}
\textcolor{blue}{\draw (1.6, -1.6) node {$\bullet$};}
\textcolor{blue}{\draw (2.4, -2.4) node {$\bullet$};}
\textcolor{blue}{\draw (3.2, -3.2) node {$\bullet$};}
\textcolor{blue}{\draw (4, -4) node {$\bullet$};}
\textcolor{blue}{\draw (-0.8, -1.6) node {$\bullet$};}
\textcolor{blue}{\draw (-0.8, -2.4) node {$\bullet$};}
\textcolor{blue}{\draw (-0.8, -3.2) node {$\bullet$};}
\textcolor{blue}{\draw (-0.8, -4) node {$\bullet$};}
\textcolor{blue}{\draw (-1.6, -2.4) node {$\bullet$};}
\textcolor{blue}{\draw (-1.6, -3.2) node {$\bullet$};}
\textcolor{blue}{\draw (-1.6, -4) node {$\bullet$};}
\textcolor{blue}{\draw (-2.4, -3.2) node {$\bullet$};}
\textcolor{blue}{\draw (-2.4, -4) node {$\bullet$};}
\textcolor{blue}{\draw (-3.2, -4) node {$\bullet$};}
\textcolor{blue}{\draw (0.8, -1.6) node {$\bullet$};}
\textcolor{blue}{\draw (0.8, -2.4) node {$\bullet$};}
\textcolor{blue}{\draw (0.8, -3.2) node {$\bullet$};}
\textcolor{blue}{\draw (0.8, -4) node {$\bullet$};}
\textcolor{blue}{\draw (1.6, -2.4) node {$\bullet$};}
\textcolor{blue}{\draw (1.6, -3.2) node {$\bullet$};}
\textcolor{blue}{\draw (1.6, -4) node {$\bullet$};}
\textcolor{blue}{\draw (2.4, -3.2) node {$\bullet$};}
\textcolor{blue}{\draw (2.4, -4) node {$\bullet$};}
\textcolor{blue}{\draw (3.2, -4) node {$\bullet$};}

% angle
\draw [domain=-40:-5] plot ({1.6*cos(\x)}, {1.6*sin(\x)});
\draw (2.7,-0.75) node {\scriptsize{$\pi/2 - \alpha$}};

\end{tikzpicture}
\end{center}
\vspace{-0.4cm}
\caption{\small Initial configuration $I(\pazocal{C}_{\alpha})$ in $\Z^2$ with type 1 sites blue and the type 2 site red.}
\label{fig:cone}
\end{figure} 

Benjamini suggested that the critical case is when the cone fills the whole negative half-plane, meaning that type 2 can survive if and only if $\alpha<\pi/2$. Our main result states, subject to an unproven assumption on the asymptotic shape for the one-type process, that this is indeed the case. In fact, the result applies in the more general setup where the cone points in an arbitrary direction $\theta$, and where the growth is governed by first-passage percolation with a general passage-time distribution. (The Richardson model corresponds to the case of exponentially distributed passage times).

\subsection{Spatial growth and competition}

We proceed to define the model of first-passage percolation, and explain its relation to the two-type Richardson model, in order to state a formal theorem. We equip the edges of the $\Z^2$ lattice with i.i.d.\ random weights $\{\omega_e\}$, referred to as passage times, and drawn from some continuous distribution supported on $[0,\infty)$ such that
\begin{equation}\label{pt_assumption}
\E\big[\min\{\omega_1,\ldots,\omega_4\}^2\big]<\infty. 
\end{equation}

The weighted graph induces a random metric on $\Z^2$. For $x,y\in\mathbb{Z}^2$, write $T(x,y)$ for the distance between $x$ and $y$ in this metric, that is, we define
\begin{equation}\label{def:T}
T(x,y):=\inf\{T(\Gamma):\Gamma\text{ is a path connecting $x$ and $y$}\},\quad\text{where }T(\Gamma):=\sum_{e\in\Gamma}\omega_e.
\end{equation}
Here and below, a `path' refers to a {\bf nearest-neighbor path} in $\Z^2$, being an alternating sequence of vertices and edges such that each edge in the sequence connects the vertices just before and after. We may unambiguously represent a nearest-neighbor path with either its vertices or its edges, and we will switch between these representations without comment.

The assumption of a continuous weight distribution ensures that there is an almost surely unique path attaining the infimum in~\eqref{def:T}. We denote the minimizing path by $\geo(x,y)$ and refer to it as the {\bf geodesic} between $x$ and $y$. The moment condition in~\eqref{pt_assumption} ensures the existence of a (deterministic) norm $\mu:\mathbb{R}^2\to [0,\infty)$, satisfying $\mu(x)>0$ for $x\neq0$, and such that
\begin{equation}\label{shapethm}
\limsup_{|x|\to\infty}\frac{|T(0,x)-\mu(x)|}{|x|}=0\quad\text{almost surely}.
\end{equation}
This is the celebrated `shape theorem' from \cite{CD81,R73}.

Let us already now mention the following strengthening of the shape theorem, proved in~\cite{AH16}, which will be useful in the proofs below: For every $\vep>0$ there exists an almost surely finite $M\ge1$ such that for all $x,y\in\Z^2$ we have
\begin{equation}\label{eq:shapethm}
\big|T(x,x+y)-\mu(y)\big|\le\vep\max\{|x|,|y|\}+M.
\end{equation}

Interpreting $T(x,y)$ as the time it takes to travel between $x$ and $y$ gives rise to a model for spatial growth. A growing entity started at the origin will at time $t$ occupy the region $\mathcal{A}_t:=\{x\in\Z^2:T(0,x)\le t\}$. Note that, if a site $x$ is occupied at time $t$, then a neighboring site $y$ will be occupied at time $t+\omega_{(x,y)}$, unless already occupied at that time. We may extend the one-type spatial growth model to a model for competing growth by considering two growing entities, initially occupying the subsets $I_1,I_2\subseteq\Z^2$, where $I_1\cap I_2=\emptyset$, and spreading through the graph according to the same principle. We allow each site to be occupied only once, and once a site is occupied by type 1 or type 2, it remains occupied by that type forever. (The assumption of a continuous passage-time distribution assures that there are no ties.) We refer to this competition process as {\bf competing first-passage percolation}. Note that the two-type Richardson model for equal strength competitors (i.e., $\lambda=1$) corresponds to the special case of exponentially distributed passage times.

The eventual fate of the competing growth model is determined by the random metric $T$. More specifically, for initial configurations $I_1$ and $I_2$, type 1 will beat type 2 to a vertex $z$, and hence occupy that vertex, if and only if $T(I_1,z)<T(I_2,z)$, where we for $I\subseteq\Z^2$ define
$$
T(I,z):=\inf_{x\in I}T(x,z).
$$
(Again, since passage times are drawn from a continuous distribution, we have $T(x,z)\neq T(y,z)$ for all $z\in\Z^2$ with probability one.) Consequently, the sets of points eventually occupied by type 1 and type 2, respectively, can be expressed as
$$
\mathcal{E}_1:=\big\{z\in\Z^2:T(I_1,z)<T(I_2,z)\big\}\quad\text{and}\quad \mathcal{E}_2:=\big\{z\in\Z^2:T(I_2,z)<T(I_1,z)\big\}.
$$
We shall say that type $i$ {\bf survives} if $\mathcal{E}_i$ is infinite, and that the two types {\bf coexist} in the case that both types survive. Understanding survival and coexistence will largely be a matter of understanding the asymptotic behaviour of the metric $T$ and its geodesics.

\subsection{Competition and the asymptotic shape}

As alluded to above, understanding geodesics in the metric space $(\Z^2,T)$ will be essential to understanding the questions of survival and coexistence. Much of our understanding of geodesics goes via the {\bf asymptotic shape}, defined as the set
$$
\mathcal{A}:=\{x\in\R^2:\mu(x)\le1\}.
$$
The relevance of the asymptotic shape comes from the fact that the statement in~\eqref{shapethm} can equivalently be formulated as a convergence result for the region $\mathcal{A}_t$ occupied at time $t$, in that $\frac{1}{t}\mathcal{A}_t\to\mathcal{A}$ as $t\to\infty$ in an appropriate sense. 
The asymptotic shape inherits all symmetries of the lattice. It follows from $\mu$ being a (non-trivial) norm that $\mathcal{A}$ is compact, convex and has non-empty interior. The asymptotic shape is widely believed to be \emph{strictly} convex and its boundary to be differentiable for a large class of passage-time distributions. However, this remains open; see \cite{ADH17} for details and references. 

Differentiability of the shape will be particularly relevant in this paper. Recall that (the boundary of) a convex set in the plane is {\bf differentiable} if there is a unique supporting line, hence referred to as a {\bf tangent}, through every point of the boundary. Let $S^1 := \{x \in \R^2: |x|=1\}$ denote the unit circle in $\R^2$. We shall throughout identify points/directions in $S^1$ with the angles $[0,2\pi)$ via the standard mapping $\theta\mapsto(\cos\theta,\sin\theta)$. We shall say that a tangent line of the asymptotic shape has {\bf direction} $\theta\in S^1$ if the tangent line intersects and is orthogonal to the half-line through the origin and $\theta$.
%We hence quantify the gradient of a tangent with the normal vector.
Note that the asymptotic shape is differentiable if and only if it has a tangent with direction $\theta$ for every $\theta\in[0,2\pi)$.

Given $\theta \in S^1$, let $\pazocal{C}_{\alpha}^{\theta}$ denote the cone with apex at the origin and angle $\alpha$ to the half-line from the origin through $-\theta$. Consider the initial configuration $I(\pazocal{C}_{\alpha}^{\theta})$ with all points in $\pazocal{C}_{\alpha}^{\theta} \setminus \{0\}$ type~1 infected and the origin type~2 infected; compare with Figure~\ref{fig:cone}.
Write $G_{\theta,\alpha}$ for the event that type~2 survives in the two-type competing first-passage percolation process started from the configuration $I(\pazocal{C}_{\alpha}^{\theta})$. That is, in the notation above, $G_{\theta,\alpha}=\{|\mathcal{E}_2|=\infty\}$ in the case that $(I_1,I_2)$ corresponds to $I(\pazocal{C}_\alpha^\theta)$.

\begin{theorem}
\label{thm1}
Assume that the passage-time distribution is continuous and satisfies \eqref{pt_assumption}. If the asymptotic shape $\mathcal{A}$ has a tangent with direction $\theta$, then $\PP(G_{\theta,\alpha}) > 0$ if and only if $\alpha < \pi/2$.
\end{theorem}

Note that, by symmetry of $\Z^2$, the shape must have at least one tangent line in each quadrant, implying that there are at least four values of $\theta$ with the property that $\PP(G_{\theta,\alpha}) > 0$ if and only if $\alpha < \pi/2$. If the asymptotic shape is differentiable, then for every direction $\theta$ we have that survival of type~2 is possible if and only if $\alpha<\pi/2$. Some related results have been obtained by L\'opez and Pimentel~\cite{LP17} for exactly solvable last-passage percolation models.

We shall prove Theorem~\ref{thm1} in two parts, which together amount to a more general picture than the one reported in the above theorem. 
We begin, in Section~\ref{sec:transition}, under a certain assumption on the geodesic structure, to identify $\pi/2$ as the critical angle below which coexistence is possible, and above which it is not (see Theorem~\ref{thm2}). Under the widely believed assumption that the asymptotic shape is differentiable, our result shows that the critical angle coincides with the flat initial condition $\alpha=\pi/2$ for all directions $\theta$. However, while we in this work have chosen to work with i.i.d.\ passage times, our approach indicate that for certain stationary and ergodic models of first-passage percolation, where the shape is known to be a polygon, there are angles $\theta$ for which the critical angle is strictly smaller than $\pi/2$. We remark further upon this, and provide references, in Section~\ref{sec:transition}.

We proceed, in Section~\ref{sec:flat_start}, to consider the critical case of a flat initial condition. Here we extend the result of Deijfen and H\"aggstr\"om~\cite{DH07}, generalised to general passage-time distributions by Antunovi\'c and Procaccia~\cite{AP17}, from coordinate directions to arbitrary directions $\theta$. This result (see Theorem~\ref{thm:flat_start}) will show that, for any value of $\theta\in[0,2\pi)$, survival of type~2 is not possible for the flat initial condition $I(\pazocal{C}_{\pi/2}^\theta)$. This argument does not require unverified assumptions on the structure of geodesics. The extension to rational directions, i.e.\ directions $\theta$ corresponding to a unit vector of the form $z/|z|$ for some $z\in\Z^2$, is relatively straightforward. Extending their result to arbitrary directions will require some care.

We begin, in Section~\ref{s:busemann}, with a discussion regarding Busemann functions, which have become an indispensable tool in order to understand the geodesic structure of $T$, and hence the possibility of coexistence in the competing growth model.

\section{Geodesics and Busemann functions}\label{s:busemann}

Coexistence in competing first-passage percolation is closely linked to the existence and structure of infinite geodesics in the first-passage metric. An {\bf infinite geodesic} is a nearest-neighbor path $(v_1,v_2,\ldots)$ such that every finite segment $(v_k,v_{k+1},\ldots,v_{k+\ell})$ is the geodesic between its endpoints $v_k$ and $v_{k+\ell}$. We shall by $\pazocal{T}_0$ denote the collection of infinite geodesics starting from the origin $0$. We shall interchangeably think of $\pazocal{T}_0$ as a set of paths as well as the graph obtained from the union of all infinite geodesics starting at the origin. Due to the assumption of continuous passage times, the resulting subgraph of $\Z^2$ is a tree, and we hence refer to $\pazocal{T}_0$ as the {\bf tree of infinite geodesics}; see Figure \ref{fig:treegeos}. A standard compactness argument shows that $|\pazocal{T}_0|$, the number of infinite geodesics starting at the origin, is always at least 1. For $v\in\Z^2$, we write $\pazocal{T}_v$ for the tree of infinite geodesics starting at $v$.

%%%%%%%%%%%%%%%%%%%%%%%%%%%%%%%%%%%%%%%%%%%%
\begin{figure}[htbp]
\begin{center}
\vspace{0.3cm}
\includegraphics[width=.32\linewidth]{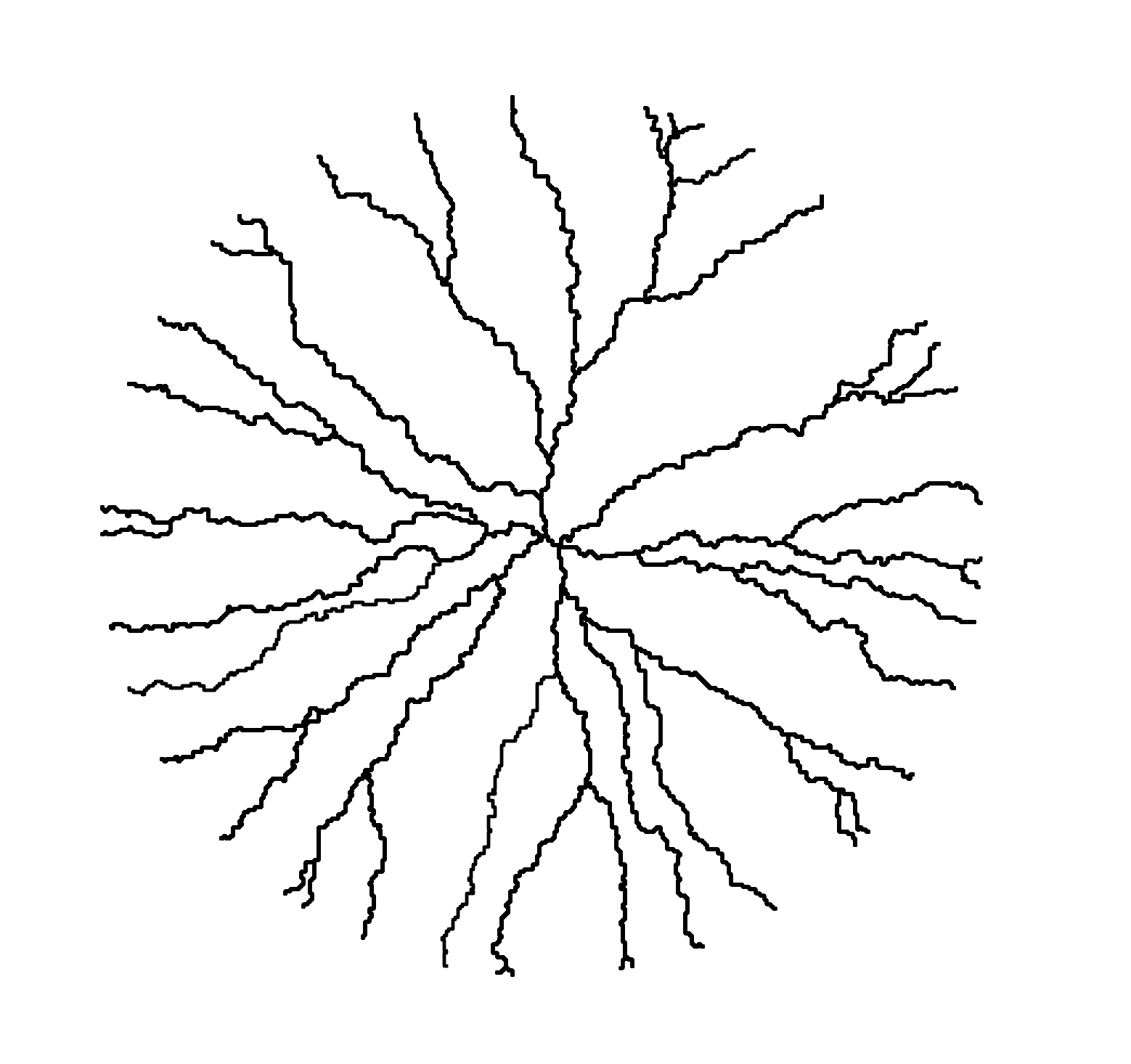}
\includegraphics[width=.32\linewidth]{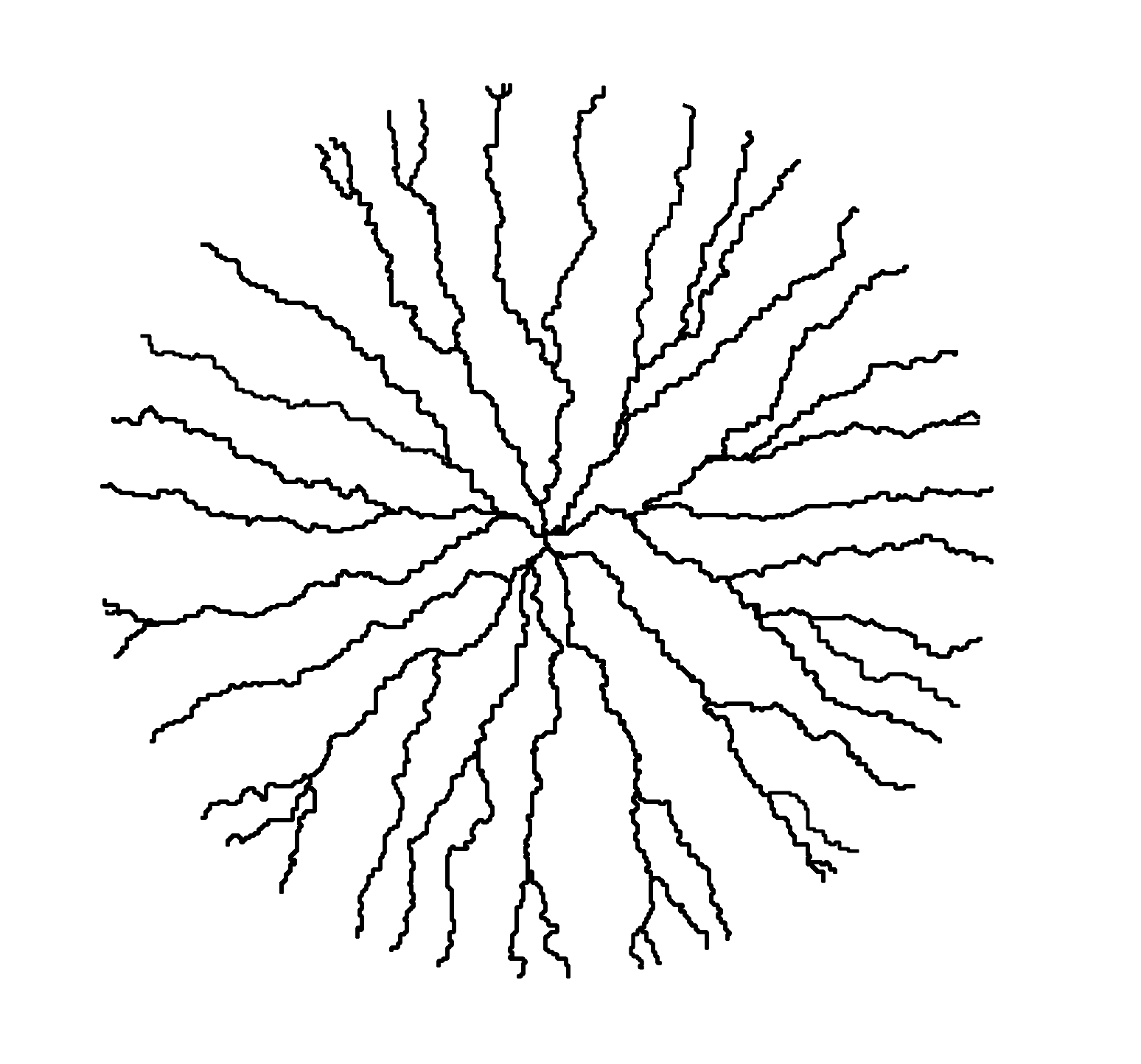}
\includegraphics[width=.32\linewidth]{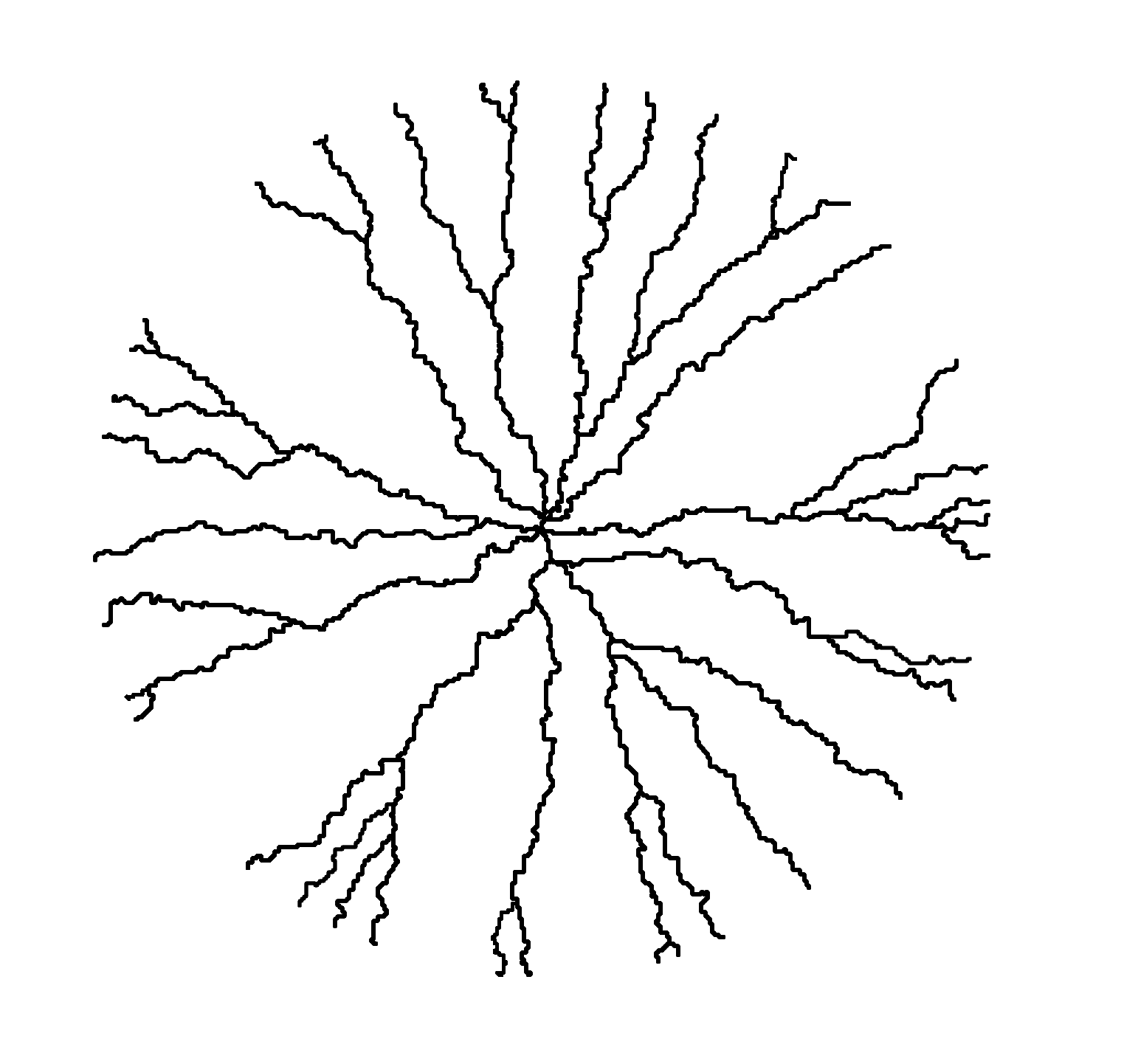}
\vspace{0.1cm}
\caption{\small Simulations of the tree of infinite geodesics $\pazocal{T}_0$ in $Z^2$.}
\label{fig:treegeos}
\end{center}
\vspace{-0.3cm}
\end{figure}
%%%%%%%%%%%%%%%%%%%%%%%%%%%%%%%%%%%%%%%%%%%%

\subsection{Geodesics and coexistence}

Infinite geodesics can be thought of as `highways to infinity' in that they define paths used by the infection to reach sites far away. That an infection started at the origin survives if it is responsible for infecting every point along some geodesic is immediate. We start with a slightly less obvious observation, which is that the infection cannot survive unless it infects all vertices along some `highway' in $\pazocal{T}_0$. 

\begin{proposition}\label{prop:geodesic}
Consider competing first-passage percolation with a continuous weight distribution, in which type~2 is initially positioned at the origin and type~1 occupies some set (finite or infinite) $I\subset\Z^2$. Then, with probability one, either type~2 infects all vertices along some geodesic $g\in\pazocal{T}_0$, or type~2 does not survive.
\end{proposition}

\begin{proof}
Let $V=\{v\in\Z^2:T(0,v)<T(u,v)\text{ for all }u\in I\}$ denote the set of vertices eventually infected by type~2. Either $V$ is infinite, or type~2 does not survive. Take $v\in V$ and consider the geodesic $\geo(0,v)$. This is the fastest path along which type~2 can reach $v$. Suppose that for some $z\in\geo(0,v)$ we have $T(0,z)\ge T(u,z)$ for some $u\in I$. Then we would also have $T(0,v)\ge T(u,v)$ for some $u\in I$, which would contradict that $v\in V$. Hence, for every $v\in V$, all vertices $z\in\geo(0,v)$ also belong to $V$.

Consider the graph with vertex set $V$ obtained by taking the union over $\geo(0,v)$ for all $v\in V$. On an event of probability one, distinct paths have distinct passage times, and if $V$ is infinite, then by K\H{o}nig's infinity lemma it contains an infinite path $g$. For every $v\in g$ the segment of $g$ from the origin to $v$ coincides with $\geo(0,v)$, so that $g\in\pazocal{T}_0$.
\end{proof}

The above proposition indicates the importance of geodesics for the question of coexistence. In order to establish coexistence, one may explore the conditions for either infection to be guaranteed survival along an infinite geodesic. So-called Busemann functions, introduced by Hoffman~\cite{H05}, are a useful tool in this context.
For every geodesic $g = (v_1, v_2, \dots)$ in $\pazocal{T}_0$ we define its {\bf Busemann function} $B_g: \Z^2 \times \Z^2 \to \R$ as the limit, for all $x,y\in\Z^2$,
$$
B_g(x,y) = \lim_{k \to \infty}\big[T(x,v_k)- T(y,v_k)\big].
$$
This limit exists almost surely for all $g\in\pazocal{T}_0$ simultaneously, as shown in~\cite{H05}.

We will below use Busemann functions in order to read out the coexistence between types in competing first-passage percolation. Again, this idea dates back to the work of Hoffman~\cite{H08}, and was elaborated upon in~\cite{A21}.
Let $g\in\pazocal{T}_0$ be a geodesic, and note that a negative value of $B_g(0,y)$ means that the origin lies `closer' than $y$ to far out vertices along $g$. That means that there exists $k\ge1$ such that, if type~1 is started at $y$ and type 2 at the origin, then type~2 will be the first to reach $v_k$, and all subsequent vertices along $g=(v_1,v_2,\ldots)$. In particular, this would imply survival of type 2 for this initial condition.

By the same argument, if type 1 initially occupies a set $I$ and type 2 initially occupies the origin, then type 2 will survive along $g$ in case that $I\subseteq\{y\in\Z^2:B_g(0,y)<0\}$. This observation will be instrumental in this paper. However, for the observation to be useful, we will need to understand some basics of how Busemann functions behave.

\subsection{Linearity of Busemann functions}

A systematic study of geodesics in first-passage percolation was initiated by Newman and collaborators in the 1990s; see \cite{LN96,LNP96,N95,NP95}. Among other things, they proved various properties about infinite geodesics, subject to the assumption that the asymptotic shape is uniformly curved -- an assumption that remains unproven until this day. These properties remain conjectural, and assert that (i) almost surely, every infinite geodesic $g=(v_1,v_2,\ldots)$ has an {\bf asymptotic direction} $\theta\in S^1$, in that $\lim_{k \to \infty} v_k/|v_k| = \theta$; (ii) for every $\theta \in S^1$, there exists an almost surely unique geodesic in $\pazocal{T}_0$ with asymptotic direction $\theta$; and (iii) for every $\theta\in S^1$, any two geodesics $g$ and $g'$ with the same asymptotic direction $\theta$ {\bf coalesce}, meaning that their symmetric difference $g \, \Delta \, g'$ is finite, almost surely. 

H\"aggstr\"om and Pemantle~\cite{HP98} proved that coexistence occurs with positive probability in the two-type Richardson model at equal strength ($\lambda=1$), and used that result to deduce that there must exist at least two infinite geodesics with positive probability. The predictions from~\cite{LN96,N95} imply that $|\pazocal{T}_0| = \infty$ almost surely, so the result from~\cite{HP98} is a modest first step towards verifying these predictions. Later work on existence of geodesics and coexistence in competing growth has adopted the reverse approach: Establish coexistence from the existence of geodesics, starting with the introduction of Busemann functions in work of Hoffman~\cite{H05,H08}.

In~\cite{H08}, Hoffman used Busemann functions as a tool to distinguish geodesics from each other, simply from the fact that if the Busemann functions of two geodesics differ, then the two geodesics must be distinct. Using Busemann functions, Hoffman could associate the existence of geodesics to tangents of the asymptotic shape $\mathcal{A}$. His work was elaborated upon by Damron and Hanson~\cite{DH14}, sharpening the connection between geodesics and properties of the asymptotic shape. Busemann functions have been an indispensable tool in the study of infinite geodesics ever since, for a large variety of first- and last-passage percolation models; see e.g.~\cite{A21,AH16,AHH22,DH14,DH17}.

In order to describe key results established in~\cite{DH14,H08}, we shall introduce some additional terminology. Recall that a linear functional $\rho: \R^2 \to \R$ is of the form $\rho=u\cdot x$, where $\cdot$ denotes scalar product, and that its level set $\{x \in \R^2: \rho(x) =1\}$ is a straight line. Also recall that a straight line is a supporting line to a convex set in $\mathbb{R}^2$ if the line contains at least one point of the set, and the whole set is contained in one of the two closed half-planes dissected by the line. We shall call a linear functional $\rho:\R^2\to\R$ {\bf supporting} if $\{x \in \R^2: \rho(x) =1\}$ is a supporting line to the asymptotic shape $\mathcal{A}$, and {\bf tangent} if $\{x \in \R^2: \rho(x) =1\}$ is the unique supporting line (the tangent line) of $\mathcal{A}$ through some point.

Given a linear functional $\rho:\R^2\to\R$ and a geodesic $g \in \pazocal{T}_0$, we say that the Busemann function of $g$ is {\bf asymptotically linear} to $\rho$ if
\begin{equation*}
 \limsup_{|y| \to \infty} \frac{1}{|y|} \big| B_g(0,y)-\rho(y)\big| = 0.
\end{equation*}
Since $B_g(0,y)=T(0,y)$ for all $y\in g$, the above can only be possible for supporting functionals. By convexity of the asymptotic shape, the set of supporting functionals can be parametrized by their gradients, and the set of supporting functionals is hence in 1-1 correspondence with the unit circle $S^1$, which in turn we have identified with the interval $[0,2\pi)$. That is, for every $\theta\in[0,2\pi)$ there exists a supporting functional $\rho$, and if the Busemann function of a geodesic $g$ is asymptotically linear to $\rho$ we shall say that the Busemann function of $g$ is asymptotically linear with {\bf direction} $\theta$.

We are now ready to formulate a first result regarding the existence of geodesics. The result is due to Damron and Hanson~\cite{DH14}, building on earlier work of Hoffman~\cite{H08}.

\begin{theorem}
\label{thmDH}
Consider first-passage percolation on $\mathbb{Z}^2$ with a continuous passage time distribution satisfying~\eqref{pt_assumption}. For every linear functional $\rho: \R^2 \to \R$ that is tangent to $\mathcal{A}$ there exists, almost surely, a geodesic in $\pazocal{T}_0$ whose Busemann function is asymptotically linear to $\rho$.
\end{theorem}

\begin{proof}
The result is a consequence of Theorem~4.3, Corollary~4.7 and Proposition~5.1 in~\cite{DH14}. See also Theorem~\ref{thmDH2} below.
\end{proof}

The theorem above, cited from~\cite{DH14}, represents a smaller part of what was proved in that paper. For the full proof of Theorem~\ref{thm1} we shall rely on a more precise formulation of what they accomplish, and will for this reason revise the notion of geodesic measures. We save, however, this discussion to Section~\ref{sec:flat_start}. We further remark that the above theorem can be reformulated in terms of so-called random coalescing geodesics, enabling the use of ergodic theory in the study of geodesics~\cite{AH16}. However, this will not be of further benefit for the purpose of this paper.

\section{Survival/non-survival for cone-shaped initial conditions}\label{sec:transition}

Having detailed the necessary preliminaries, we are ready to proceed to the proof of Theorem~\ref{thm1}. The proof will be divided into two separate statements. The first is stated and proved in this section, and identifies $\pi/2$ as the critical angle, subject to a widely believed assumption on the existence of geodesics. The critical case, for a flat initial condition, is saved to the next section.

\begin{theorem}\label{thm2}
Consider competing first-passage percolation on $\mathbb{Z}^2$, with continuous passage times satisfying~\eqref{pt_assumption}, and with initial configuration $I(\pazocal{C}_{\alpha}^{\theta})$, for some $\theta\in[0,2\pi)$. Suppose that, with probability one, there exists a geodesic in $\pazocal{T}_0$ whose Busemann function is asymptotically linear with direction $\theta$. Then, $\PP(G_{\theta,\alpha})>0$ for $\alpha<\pi/2$ and $\PP(G_{\theta,\alpha})=0$ for $\alpha>\pi/2$.
\end{theorem}

Under the widely believed conjecture that the asymptotic shape is differentiable it would follow that the critical angle equals $\pi/2$ for all directions $\theta$. However, it is also known that there exist stationary and ergodic models of first-passage percolation for which the shape is not differentiable; see~\cite{AB18,BH21,HM95}.
In the model of~\cite{AB18}, the asymptotic shape is an octagon, and although this model does not fit into our setting, our argument can be adapted to show that there exists a critical angle $\alpha_c=\alpha_c(\theta)$ for coexistence which in this case satisfies $\alpha_c<\pi/2$ for all diagonal directions, i.e.\ $\theta\in\{\pi/4,3\pi/4,5\pi/4,7\pi/4\}$.
In the model of~\cite{BH21}, the asymptotic shape is a diamond, and the critical angle $\alpha_c=\alpha_c(\theta)$ equals $\pi/4$ in the case that $\theta$ corresponds to either of the coordinate directions, i.e.\ $\theta\in\{0,\pi/2,\pi,3\pi/2\}$.

We also remark that, due to lattice symmetry and convexity of the shape, whenever the critical angle exists, it satisfies $\pi/4\le\alpha_c(\theta)\le\pi/2$ regardless of the direction $\theta$ for models that satisfy the conditions of Damron and Hanson~\cite{DH14}. This is because there exist at least four tangent directions of the shape, and hence at least four directions with linear Busemann function, separated by an angle $\pi/2$. That means that the largest directional gap without a linear Busemann function is $\pi/2$. Consequently, regardless of the direction $\theta$, and without non-verified assumptions on the geodesic structure, one can show that survival is always possible for $\alpha<\pi/4$. We leave it to the reader to fill in the details of this argument.

Heuristically, Theorem~\ref{thm2} can be explained as follows. Suppose, for simplicity, that $\theta=\pi/2$, so that the linear functional $\rho$, to which $B_{g}$ is asymptotically linear, is zero on the horizontal axis. If $\alpha<\pi/2$, then $\rho(x)<0$ for all points in the cone $\mathcal{C}_{\alpha}^\theta$ excluding the origin, meaning that the points are (on average) asymptotically further away than the origin from $g$, so that type~2 at the origin should have a chance to reach vertices along $g$ before type~1. If, on the other hand, $\alpha>\pi/2$ then there are points in the right half-plane contained inside $\mathcal{C}_{\alpha}^\theta$ for which $\rho(x)>0$, that are hence asymptotically closer (on average) than the origin to far out vertices along $g$. Similarly, there are points in the left half-plane contained in $\mathcal{C}_\alpha^\theta$ that also should beat type~2 to far out vertices along $g$. Hence, if $\alpha>\pi/2$, then the type~2 infection at the origin ought to become trapped by type~1 infections surrounding it from both sides.

We split the proof in two subsections. First we show that $\PP(G_{\theta,\alpha})=0$ for $\alpha>\pi/2$, and then that $\PP(G_{\theta,\alpha})>0$ for $\alpha<\pi/2$. For the former part, it is straightforward to turn the outlined heuristic into a proof. For the latter, we will have to perform a resampling argument that will require some work.

\subsection{Non-survival above the critical angle}\label{sec:super}

Fix $\theta\in[0,2\pi)$. The straight line through the origin with direction $\theta$ divides the plane into two half-planes that we denote as follows:
\begin{align*}
\mathbb{H}_+(\theta)&:=\big\{x\in\R^2:x\cdot(\sin\theta,-\cos\theta)\ge0\big\};\\
\mathbb{H}_-(\theta)&:=\big\{x\in\R^2:x\cdot(\sin\theta,-\cos\theta)\le0\big\}.
\end{align*}
In addition, let $L_-(\theta):=\{r(\cos\theta,\sin\theta):r\le0\}$ denote the negative half-line in direction $\theta$, and note that the cone $\pazocal{C}_\alpha^\theta$ is centred around this half-line.

Suppose that, almost surely, there exists a geodesic $g\in\pazocal{T}_0$ whose Busemann function is asymptotically linear with direction $\theta$. Given $g$, we partition the plane into two parts according to its Busemann function. In particular, we let
$$
Z(g):=\big\{y\in\Z^2:B_g(0,y)>0\big\}
$$
denote the set consisting of all vertices that lie `asymptotically closer' to far out vertices along the geodesic $g$ in comparison to the origin. (It is a priori not clear whether $B_g(0,y)=0$ can occur with positive probability, but that shall not be of any concern for our argument.)

Fix $\alpha>\pi/2$ and set $\alpha':=(\alpha+\pi/2)/2$. Let $A_+$ denote the event that the intersection of the sets $Z(g)$ and $(\pazocal{C}_\alpha^\theta\setminus\pazocal{C}_{\alpha'}^\theta)\cap \mathbb{H}_+(\theta)$ is infinite. By assumption, the Busemann function of $g$ is asymptotically linear with direction $\theta$, almost surely, implying that $\PP(A_+)=1$.

Next, given $\vep>0$, let $A'$ denote the event that $T(x,y)>\vep|y|$ for all $x$ within distance 1 of the half-line $L_-(\theta)$ and all but finitely many $y$ in $\pazocal{C}_\alpha^\theta\setminus\pazocal{C}_{\alpha'}^\theta$.
By the strengthening of the shape theorem, in~\eqref{eq:shapethm}, we may pick $\vep>0$ such that $\PP(A')=1$.
Finally, since, almost surely, the Busemann function of $g$ is asymptotically linear with direction $\theta$, it follows that
$$
M':=\max\big\{B_{g}(0,y):y\in\Z^2\text{ at distance at most 1 from }L_-(\theta)\big\}
$$
is almost surely finite. Set $A'':=\{M'<\infty\}$.

Now, on the event $A_+\cap A'\cap A''$, which occurs with probability 1, we may find $u'$ in the region $(\pazocal{C}_\alpha^\theta\setminus\pazocal{C}_{\alpha'}^\theta)\cap \mathbb{H}_+(\theta)$ such that $B_{g}(0,u')>0$ and $T(x,u')>M'$ for all $x$ at distance at most 1 to $L_-(\theta)$. Consequently, there exists $z'\in g$ such that $T(u',z')<T(0,z')$, and the path $\gamma'$ connecting $u'$ to $z'$ satisfying $T(\gamma')=T(u',z')$ does not cross $L_-(\theta)$; see Figure~\ref{fig:supercritical}. 
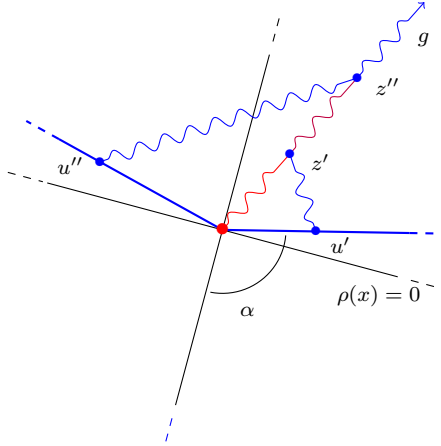
\begin{figure}[htbp]
\begin{center}
\begin{tikzpicture}[scale=.6, rotate=-15,
       % ->,    > = stealth',
       %shorten > = 1pt,auto,
       decoration = {snake, pre length=3pt,post length=4pt}]
% axes
\draw (-4,0) -- (4,0); % x axis
\draw (0,-4) -- (0,4); % y axis
% dashed lines
\draw[dashed] (-4.9,0) -- (-4.1,0);
\draw[dashed] (4.3,0) -- (4.9,0);
\draw[dashed] (0, 4.3) -- (0, 4.9);
\draw[draw = blue, dashed] (0, -4.3) -- (0, -4.9);

% geodesics
\path[draw=red, decorate] (0,0) --  (1,2);
%\path[draw=red, decorate] (0,0) --  (-1,2);
\draw (3.2, 5.2) node {\scriptsize{$g$}};
%\draw (-2.25, 3.2) node {\scriptsize{$g_2$}};
\path[draw = purple, decorate] (1,2) --  (2,4);
%\path[draw = blue, ->, decorate] (-1,2) --  (-2,4);
\path[draw = blue, ->, decorate] (2,4) --  (3,6);

% rho
%\draw (-4,2) -- (4,-2);
%\draw (-4,-2) -- (4,2);
\draw (3.7, -.5) node {\scriptsize{$\rho(x)=0$}};
%\draw (-3.7, -2.5) node {\scriptsize{$\rho_2(x)=0$}};

%cone
\draw[draw = blue, thick] (-4,1) -- (0,0) -- (4,1);
\draw[draw = blue, dashed, thick] (-4,1) -- (-4.8,1.2);
\draw[draw = blue, dashed, thick] (4,1) -- (4.8,1.2);
\textcolor{blue}{\draw (2, 0.5) node {\scriptsize{$\bullet$}};}
\textcolor{blue}{\draw (-3, 0.75) node {\scriptsize{$\bullet$}};}
\draw (2.6, 0.35) node {\scriptsize{$u'$}};
\draw (-3.6, 0.5) node {\scriptsize{$u''$}};

% paths
\textcolor{blue}{\draw (1, 2) node {\scriptsize{$\bullet$}};}
\textcolor{blue}{\draw (2, 4) node {\scriptsize{$\bullet$}};}
\draw (1.7, 2) node {\scriptsize{$z'$}};
\draw (2.7, 4) node {\scriptsize{$z''$}};
\path[draw = blue, decorate] (2,0.5) --  (1,2);
\path[draw = blue, decorate] (-3,0.75) --  (2,4);

% sites
\textcolor{red}{\draw (0,0) node {$\bullet$};}

% angle
\draw [domain=-86:10] plot ({1.4*cos(\x)}, {1.4*sin(\x)});
\draw (1,-1.6) node {\scriptsize{$\alpha$}};
%\draw [domain=-87:-17] plot ({2.6*cos(\x)}, {2.6*sin(\x)});
%\draw (1.5,-2.5) node {\scriptsize{$\alpha$}};

\end{tikzpicture}
\end{center}
\vspace{-0.4cm}
\caption{\small The supercritical case with $\alpha>\pi/2$.}
\label{fig:supercritical}
\end{figure}
It follows that in the competition process where the origin is initially infected by type~2 and remaining vertices in $\pazocal{C}_\alpha^\theta$ are infected by type~1, type~1 will infect the vertex $z'$ as well as all subsequent vertices along $g$ and $\gamma'$. Since the path $\gamma'$ does not intersect $L_-(\theta)$ it effectively blocks type~2 from infecting any vertex in $\mathbb{H}_+(\theta)\setminus\pazocal{C}_\alpha^\theta$ that lies clockwise of the concatenation of $\gamma'$ and the segment of $g$ starting at $z'$.

An analogous argument shows that we may, with probability 1, find $u''$ in $(\pazocal{C}_\alpha^\theta\setminus\pazocal{C}_{\alpha''}^\theta)\cap \mathbb{H}_-(\theta)$ and $z''\in g$ such that $T(u'',z'')<T(0,z'')$, and the path $\gamma''$ connecting $u''$ to $z''$ satisfying $T(\gamma'')=T(u'',z'')$ does not cross $L_-(\theta)$; again see Figure~\ref{fig:supercritical}. In the competition process where the origin is initially infected by type~2 and remaining vertices in $\pazocal{C}_\alpha^\theta$ are infected by type~1, we conclude again that type~1 will infect the vertex $z''$ as well as all subsequent vertices along $g$ and $\gamma''$. Since the path $\gamma''$ does not intersect $L_-(\theta)$ it effectively blocks type~2 from infecting any vertex in $\mathbb{H}_-(\theta)\setminus\pazocal{C}_\alpha^\theta$ that lies counterclockwise of the concatenation of $\gamma''$ and the segment of $g$ starting at $z''$.

In conclusion, if $\alpha>\pi/2$, then, almost surely, type~2 will be first to at most finitely many vertices along $g$, as well as any other geodesic. This means that, almost surely, for every infinite geodesic starting at the origin, type~2 will infect at most finitely many vertices. By Proposition~\ref{prop:geodesic}, this means that type~2 will infect at most finitely many vertices altogether, and hence will not survive, almost surely. This finishes the proof of the non-survival part of Theorem~\ref{thm2}.

\subsection{Survival below the critical angle} \label{sec:sub}

It remains to prove the part of Theorem~\ref{thm2} that states that survival is possible for $\alpha<\pi/2$. By lattice symmetry, we may assume that $\theta\in[\pi/4,\pi/2]$.
Assume that with probability one there exists a geodesic $g$ in $\pazocal{T}_0$ with Busemann function asymptotically linear with direction $\theta$. 
For $k \ge1$, denote by $\pazocal{E}_k=\pazocal{E}_k(\alpha)$ the event that the Busemann function $B_{g}$ is negative on the cone $\pazocal{C}_\alpha^\theta$ below height level $-k$. More precisely, denote by
$R_k=R_k(\alpha)$ the region obtained by subtracting from the cone $\pazocal{C}_\alpha^\theta$ with infinite height the cone with the same apex and angle with finite height $k$,  depicted in Figure~\ref{fig:subcritical} (left), and set
\begin{equation}\label{eq:E_k}
\pazocal{E}_k(\alpha) := \Big\{ \sup_{x \in R_k} B_{g}(0,x) < 0 \Big\}.
\end{equation}
For $\alpha<\pi/2$, the zero set of the linear functional $\rho$, to which $B_g$ is asymptotically linear, intersects $\pazocal{C}_\alpha^\theta$. Since the Busemann function of $g$ is asymptotically linear to $\rho$, almost surely, we may for every $\alpha<\pi/2$ and $\vep>0$ find $K<\infty$ such that $\PP(\pazocal{E}_k(\alpha)) \geq 1- \vep$ for all $k\ge K$.

Note that on the event $\pazocal{E}_k(\alpha)$ we have for every $z\in g$ that $T(0,z)<\inf_{x\in R_k}T(x,z)$, and hence that type~2 would conquer all vertices along $g$ in the competition problem where type~2 starts at the origin and type~1 from the vertices in $R_k$. That is, type~2 would be able to escape to infinity along $g$ when competing against type~1 with initial set restricted to $R_k$.

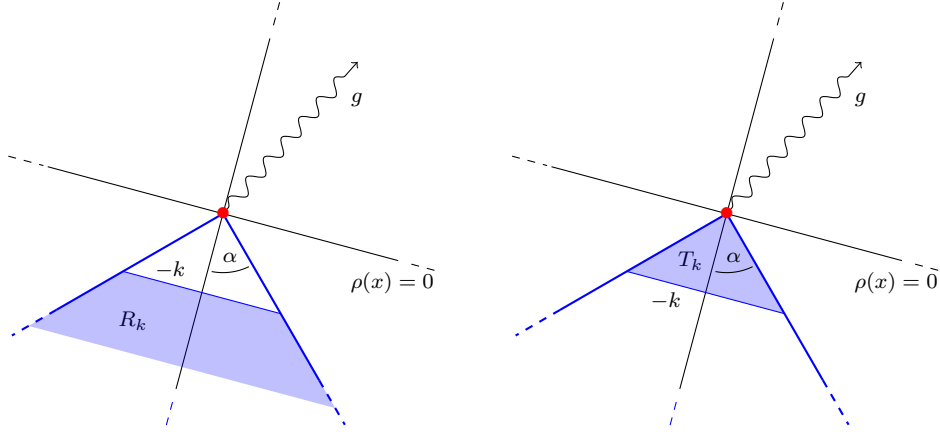
\begin{figure}[htbp]
\begin{center}
\begin{tikzpicture}[scale=.6, rotate=-15,
       % ->,    > = stealth',
       %shorten > = 1pt,auto,
       decoration = {snake, pre length=3pt,post length=4pt}]
% axes
\draw (-4,0) -- (4,0); % x axis
\draw (0,-4) -- (0,4); % y axis
% dashed lines
\draw[dashed] (-4.9,0) -- (-4.1,0);
\draw[dashed] (4.3,0) -- (4.9,0);
\draw[dashed] (0, 4.3) -- (0, 4.9);
\draw[draw = blue, dashed] (0, -4.3) -- (0, -4.9);

% geodesics
\path[draw, ->, decorate] (0,0) --  (2,4);
%\path[draw, ->, decorate] (0,0) --  (-2,4);
\draw (2.2, 3.2) node {\scriptsize{$g$}};
%\draw (-2.25, 3.2) node {\scriptsize{$g_2$}};

% rho
%\draw (-4,2) -- (4,-2);
%\draw (-4,-2) -- (4,2);
\draw (4, -.5) node {\scriptsize{$\rho(x)=0$}};
%\draw (-4, -2.5) node {\scriptsize{$\rho_2(x)=0$}};

%cone
\draw[draw = blue, thick] (-3,-3) -- (0,0) -- (3,-3);
\draw[draw = blue, dashed, thick] (-3,-3) -- (-3.8,-3.8);
\draw[draw = blue, dashed, thick] (3,-3) -- (3.8,-3.8);
\draw[draw = blue] (-1.8,-1.8) -- (1.8,-1.8);
\draw (-0.8, -1.5) node {\scriptsize{$-k$}};
\draw (-1.3, -2.8) node {\scriptsize{$R_k$}};
\fill[blue,nearly transparent] (-3.5, -3.5) -- (-1.8,-1.8) -- (1.8,-1.8) -- (3.5,-3.5) -- cycle;

% sites
\textcolor{red}{\draw (0,0) node {$\bullet$};}

% angle
\draw [domain=-86:-49] plot ({1.3*cos(\x)}, {1.3*sin(\x)});
\draw (0.4,-0.9) node {\scriptsize{$\alpha$}};

%\draw [domain=-88:-2] plot ({3*cos(\x)}, {3*sin(\x)});
%\draw (1.1,-2.4) node {\scriptsize{$\pi/2$}};

\end{tikzpicture}
\quad\quad
\begin{tikzpicture}[scale=.6, rotate=-15,
       % ->,    > = stealth',
       %shorten > = 1pt,auto,
       decoration = {snake, pre length=3pt,post length=4pt}]
% axes
\draw (-4,0) -- (4,0); % x axis
\draw (0,-4) -- (0,4); % y axis
% dashed lines
\draw[dashed] (-4.9,0) -- (-4.1,0);
\draw[dashed] (4.3,0) -- (4.9,0);
\draw[dashed] (0, 4.3) -- (0, 4.9);
\draw[draw = blue, dashed] (0, -4.3) -- (0, -4.9);

% geodesics
\path[draw, ->, decorate] (0,0) --  (2,4);
%\path[draw, ->, decorate] (0,0) --  (-2,4);
\draw (2.2, 3.2) node {\scriptsize{$g$}};
%\draw (-2.25, 3.2) node {\scriptsize{$g_2$}};

% rho
%\draw (-4,2) -- (4,-2);
%\draw (-4,-2) -- (4,2);
\draw (4, -.5) node {\scriptsize{$\rho(x)=0$}};
%\draw (-4.1, -2.5) node {\scriptsize{$\rho_2(x)=0$}};

%cone
\draw[draw = blue, thick] (-3,-3) -- (0,0) -- (3,-3);
\draw[draw = blue, dashed, thick] (-3,-3) -- (-3.8,-3.8);
\draw[draw = blue, dashed, thick] (3,-3) -- (3.8,-3.8);
\draw[draw = blue] (-1.8,-1.8) -- (1.8,-1.8);
\draw (-0.8, -2.2) node {\scriptsize{$-k$}};
\draw (-0.5, -1.2) node {\scriptsize{$T_k$}};
\fill[blue,nearly transparent] (-1.8,-1.8) -- (0,0) -- (1.8,-1.8) -- cycle;

% sites
\textcolor{red}{\draw (0,0) node {$\bullet$};}

% angle
\draw [domain=-86:-49] plot ({1.3*cos(\x)}, {1.3*sin(\x)});
\draw (0.4,-0.9) node {\scriptsize{$\alpha$}};

%\draw [domain=-88:-29] plot ({3*cos(\x)}, {3*sin(\x)});
%\draw (1.1,-2.4) node {\scriptsize{$\alpha_1$}};

\end{tikzpicture}
\end{center}
\vspace{-0.4cm}
\caption{\small The subcritical case $\alpha < \pi/2$, with the regions $R_k$ (left) and $T_k$ (right) coloured in blue.}
\label{fig:subcritical}
\end{figure}

Next, let $T_k=T_k(\alpha)$ be the triangular region obtained by subtracting $R_k$ from $\pazocal{C}_\alpha^\theta$; see Figure~\ref{fig:subcritical} (right).
%For general $\theta$, again we tilt the cone, including the region $T_k$, to point in direction $\theta$.
We need to show that type~2 has positive probability of escaping to infinity along $g$ when competing simultaneously against type~1 starting from sites in $R_k$ and $T_k$. In the case that the passage time distribution has unbounded support, this can be done by a rather straightforward local modification argument, in which the passage times of edges connecting vertices in $T_k\setminus\{0\}$ with vertices in its complement are increased substantially, impeding type~1 from advancing. In order to avoid imposing a restriction on the support of the distribution we shall require a somewhat more elaborate modification argument.

We now fix $\alpha<\pi/2$, and fix $a<b$ such that $\PP(\omega_e<a)>0$ and $\PP(\omega_e>b)>0$. Set $\Delta=(b-a)/2$. Pick $\beta\in(\pi/6,\pi/4)$ such that $b\tan(\beta)>a+\Delta$, and $\alpha'\in(\alpha,\pi/2)$ such that $\alpha'>\pi/4+\beta$. 
We also fix $c\ge1$ such that
\begin{equation}\label{eq:c_pick}
\frac{\mu({\bf e}_1)}{12}\Big[1+3c\frac{\tan(\alpha'-\alpha)}{\cos(\alpha)}\Big]>a+\Delta.
\end{equation}

For $n\ge0$, we shall below write $\pazocal{C}^\theta_{\alpha'}(n)$ for the cone $\pazocal{C}^\theta_{\alpha'}$ shifted along the vector $n{\bf e}_2$, so that its apex lies at $n{\bf e}_2$, and write $T_k(\alpha',n)$ for the region $T_k(\alpha')$ shifted along the vector $n{\bf e}_2$; see Figure~\ref{fig:coneinsidecone}.
We also write $\pazocal{E}_k(\alpha',n)$ for the translate of the event $\pazocal{E}_k(\alpha')$ along the vector $n{\bf e}_2$, and fix $k$ so that $\PP(\pazocal{E}_k(\alpha',n))\ge4/5$. Given $n$ and $k$ we let $n'$ denote the maximal integer so that $n'{\bf e}_2$ does not belong to the interior of $T_k(\alpha',n)$, and let $\gamma_n$ denote the straight path connecting $n'{\bf e}_2$ and $n{\bf e}_2$, and pick $K$ so that $\PP(T(\gamma_n)\le K)\ge4/5$.
%$$
%\PP\big(\pazocal{E}_k(\alpha',n)\cap\{T(\gamma_n)\le K\}\big)\ge3/4.
%$$
For $\alpha$, $\alpha'$ and $k$ fixed as above, let
$$
\pazocal{A}_n:=\big\{T(u,v)>\mu(v-u)/2\text{ for all }u\in\pazocal{C}^\theta_\alpha,v\in\partial\pazocal{C}^\theta_{\alpha'}(n')\big\}.
$$
Since $\alpha<\alpha'$ we may, due to the strengthening of the shape theorem in~\eqref{eq:shapethm}, fix $N$ large so that $N'\Delta>K$ and $\PP(\pazocal{A}_N)\ge4/5$, which hence gives that
\begin{equation}\label{eq:good_event}
\PP\big(\pazocal{E}_k(\alpha',N)\cap\{T(\gamma_N)\le K\}\cap\pazocal{A}_N\big)\ge2/5.
\end{equation}

For $c$, $k$ and $N'$ specified as above, let $D:=T_{(1+2c)N'}(\alpha',N')$; see Figure~\ref{fig:coneinsidecone}. Our goal for the remainder of the proof will be to increase the weight along edges in $D$, except for those connecting the origin to $N'{\bf e}_2$, which will allow for type~2, started at the origin, to reach $N'{\bf e}_2$ with $K$ time units to spare before type~1, started in the remainder of $\pazocal{C}^\theta_\alpha$, reaches a boundary point of $\pazocal{C}^\theta_{\alpha'}(N')$. On this event type~2 will succeed in escaping to infinity along a geodesic starting at $N{\bf e}_2$.
(For simplicity we have let the region $D$ intersect $\pazocal{C}_\alpha^\theta$. Note, however, that modifying edges between vertices in $\pazocal{C}_\alpha^\theta$ has no effect on the competition between the two types.)

\begin{figure}[htbp]
\begin{center}
\begin{tikzpicture}[scale=.5, 
       % ->,    > = stealth',
       %shorten > = 1pt,auto,
       decoration = {snake, pre length=3pt,post length=4pt}]
      
% axes
\draw (-11,0) -- (7,0); % x axis
\draw (0,-5) -- (0,8.5); % y axis
\draw[dashed] (-11.7,0) -- (-11.1,0);
\draw[dashed] (7.2,0) -- (7.8,0);
\draw[dashed] (0, 8.6) -- (0, 9.2);
\draw[dashed] (0, -5.3) -- (0, -5.9);

% geodesic
\path[draw, ->, decorate] (0,5) --  (2,9);

% cone above
\begin{scope}[rotate around={-15:(0,5)}]
\draw[draw = black] (-9,-3) -- (0,5) -- (9,-3);
\draw[draw = black, dashed, ] (-9.9,-3.8) -- (-9,-3);
\draw[draw = black, dashed, ] (9.9,-3.8) -- (9,-3);
\fill[orange,nearly transparent] (-1.13, 4) -- (1.13,4) -- (0,5) -- cycle;
\draw[draw=black] (-1.13,4) -- (1.13,4);
\end{scope}

% cone middle
\begin{scope}[rotate around={-15:(0,4)}]
\draw (0,-5) -- (0,5);
\draw[draw = black] (-8,-3) -- (0,4) -- (8,-3);
\draw[draw = black, dashed, ] (-8.9,-3.8) -- (-8,-3);
\draw[draw = black, dashed, ] (8.9,-3.8) -- (8,-3);
\draw[domain=-135:-94] plot ({1.4*cos(\x)}, {1.4*sin(\x) +4});
\draw (-0.4,3.1) node {\scriptsize{$\alpha'$}};
\draw[domain=-72:-45] plot ({1.4*cos(\x)}, {1.4*sin(\x) +4});
\draw (0.55,3.1) node {\scriptsize{$\delta$}};
\draw[draw = black] (-7.2, -2.3) -- (-1.4,-2.3);
\draw[draw=black] (3.5,-2.3) -- (7.2, -2.3);
%\draw[draw = black] (8.2, -2.3) -- (2.3, -2.3);
%\draw (-1.5, -2.275) node {\scriptsize{$-2cN$}};
\fill[YellowGreen,nearly transparent] (-7.2, -2.3) -- (-1.4,-2.3) -- (1,0.15) -- (3.5,-2.3) -- (7.2, -2.3) -- (0,4) -- cycle;
%\draw (-2.6, 0.8) node {\scriptsize{$D$}};
%\fill[YellowGreen,nearly transparent] (8.2, -2.3) -- (2.7,2.6) -- (0,2.6) -- (0,0) --(2.3, -2.3) -- cycle;
\draw (-2.6, 0) node {\scriptsize{$D$}};
%\draw[draw = ForestGreen, thick] (0,0) -- (0,2.6);
%\draw (0.6, 1.5) node {\scriptsize{$\gamma'$}};
\end{scope}

%cone below
\begin{scope}[rotate=-15]
%\fill[white] (-3.5, -3.5) -- (0,0) -- (3.5, -3.5) -- cycle;
\draw (0,-5) -- (0,5);
\draw[draw = blue, thick] (-3,-3) -- (0,0) -- (3,-3);
\draw[draw = blue, dashed, thick] (-3,-3) -- (-3.8,-3.8);
\draw[draw = blue, dashed, thick] (3,-3) -- (3.8,-3.8);
\fill[blue, nearly transparent] (-3.5, -3.5) -- (0,0)-- (3.5,-3.5) -- cycle;

% angle below
\draw [domain=-131:-94] plot ({1.3*cos(\x)}, {1.3*sin(\x)});
\draw (-0.4,-0.9) node {\scriptsize{$\alpha$}};
\end{scope}

\draw [domain=4:71] plot ({1*cos(\x)}, {1*sin(\x)});
\draw (0.5,0.4) node {\scriptsize{$\theta$}};

% labels
\draw (1.8, 5) node {\scriptsize{$N$}};
\draw (1.8, 4) node {\scriptsize{$N'$}};
\draw (-8.5, 1.6) node {\scriptsize{$\pazocal{C}_{\alpha'}^\theta(N)$}};
\draw (-8.5, -1.2) node {\scriptsize{$\pazocal{C}_{\alpha'}^\theta(N')$}};
\draw (-3.5, -3.5) node {\scriptsize{$\pazocal{C}_{\alpha}^\theta$}};
\draw (1.6,-.4) node {\scriptsize{$x$}};

% sites
\textcolor{red}{\draw (0,0) node {$\bullet$};}
\textcolor{black}{\draw (0,5) node {$\bullet$};}
\textcolor{black}{\draw (0,4) node {$\bullet$};}

\end{tikzpicture}
\end{center}
\vspace{-0.4cm}
\caption{\small The positioning of the three cones $\pazocal{C}_{\alpha}^\theta$, $\pazocal{C}_{\alpha'}^\theta(N')$ and $\pazocal{C}_{\alpha'}(N)$. The three dots indicate the origin, and the points $N'{\bf e}_2$ and $N{\bf e}_2$. The three shaded areas indicate the starting configuration of type~1, and the regions $D$ and $T_k(\alpha',N)$. }
\label{fig:coneinsidecone}
\end{figure}
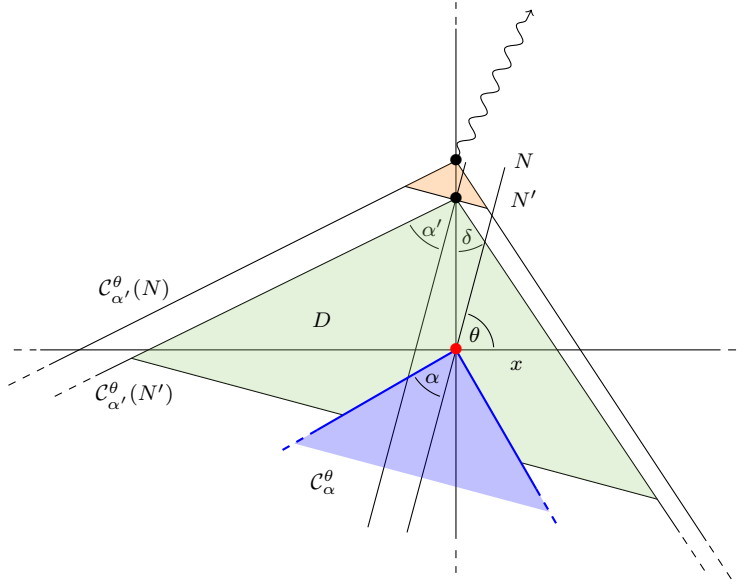

We proceed to the modification argument. First we pick $b'\ge b$ such that $\PP(\omega_e>b')>0$ and
\begin{equation}\label{eq:good_weights}
\PP\big(\omega_e\le b'\text{ for all }e\in D\big)\ge4/5,
\end{equation}
which is possible since the weight distribution is assumed to be continuous. 
Let $\omega'$ be an independent copy of $\omega$, and construct a third configuration $\tilde\omega$ by identifying $\tilde\omega$ with $\omega'$ on $D$ and identifying $\tilde\omega$ with $\omega$ elsewhere. Let $\pazocal{G}$ denote the intersection of the events in~\eqref{eq:good_event} and~\eqref{eq:good_weights}, which hence occurs with probability at least $1/5$.
Let $\pazocal{G}'$ denote the event that $\omega'_e\le a$ for the edges along the vertical axis connecting the origin to $N'{\bf e}_2$, and that $\omega'_e\ge b'$ for all remaining edges in $D$. By the choice of $a$ and $b'$ we have $\PP(\pazocal{G}')>0$, and by the independence of $\omega$ and $\omega'$ we have
$$
\PP(\pazocal{G}\cap\pazocal{G}')=\PP(\pazocal{G})\PP(\pazocal{G}')>0.
$$

In the remainder of the proof we consider an outcome $(\omega,\omega')$ in the event $\pazocal{G}\cap\pazocal{G}'$, and write $\tilde T$ for passage times in the resulting weight configuration $\tilde\omega$. Let $\sigma_\ell$ denote the straight path connecting the origin and $\ell{\bf e}_2$. For every $\ell=1,2,\ldots,N'$ we have, on the event $\pazocal{G}\cap\pazocal{G}'$, that $\tilde T(\sigma_\ell)\le a\ell$ in the configuration $\tilde\omega$. Now consider a path $\gamma$, edge-disjoint from $\sigma_\ell$, connecting some point in $\pazocal{C}_\alpha^\theta$ to $\ell{\bf e}_2$, for some $\ell=1,2,\ldots,N'$. Either $\gamma$ exits $\pazocal{C}_{\alpha'}^\theta(N')$ or it does not. If $\gamma$ does not exit $\pazocal{C}_{\alpha'}^\theta(N')$, then it necessarily includes at least $\ell$ edges which in $\tilde\omega$ has weight $b$ or more. That is, for each $\ell=1,2,\ldots,N'$, and any path $\gamma$ disjoint from $\sigma_\ell$ that does not exit $\pazocal{C}_{\alpha'}^\theta(N')$, we have in $\tilde\omega$ that
\begin{equation}\label{eq:gamma1}
\tilde T(\gamma)\ge b\ell>a\ell\ge \tilde T(\sigma_\ell).
\end{equation}
In conclusion, for every $\ell=1,2,\ldots,N'$, type~1 cannot beat type~2 to $\ell{\bf e}_2$ staying within the cone $\pazocal{C}_{\alpha'}^\theta(N')$.

We consider next a path $\gamma$ connecting a point $u\in\pazocal{C}_\alpha^\theta$ to the exterior of $\pazocal{C}_{\alpha'}^\theta(N')$. Since $\alpha<\alpha'$, the point in $\pazocal{C}_\alpha^\theta$ at closest distance to the exterior of $\pazocal{C}_{\alpha'}^\theta(N')$ is the origin. Since $\theta\in[\pi/4,\pi/2]$ and $\alpha'\in(\pi/4,\pi/2)$, the lattice point outside $\pazocal{C}_{\alpha'}^\theta(N')$ at closest $\ell_1$-distance to the origin is either the point $N'{\bf e}_2$ itself, or a point lying on the positive half of the horizontal axis, depending on whether the angle $\delta:=\alpha'-\pi/2+\theta$ (depicted in Figure~\ref{fig:coneinsidecone}) exceeds $\pi/4$ or is less than $\pi/4$, respectively. In the former case the path has length $N'$, and in the latter case it has length (denoted by $x$ in Figure~\ref{fig:coneinsidecone}) at least
$$
N'\tan(\delta)=N'\tan(\alpha'-\pi/2+\theta)\ge N'\tan(\alpha'-\pi/4)\ge N'\tan(\beta),
$$
where we have used that $\pi/2-\theta\le\pi/4$ and $\alpha'-\pi/4\ge\beta$. If the starting point $u$ of $\gamma$ is contained in $T_{cN'}(\alpha)$, then $\gamma$ will again have to pick up at least $N'\tan\beta$ edges which in $\tilde\omega$ has weight at least $b$, so that
\begin{equation}\label{eq:gamma2}
\tilde T(\gamma)\ge N'b\tan(\beta)>N'(a+\Delta)>N'a+K,
\end{equation}
where we have used that $b\tan\beta>\alpha+\Delta$ and that $N'\Delta> K$.
In particular, along no such path is type~1 able to reach $\ell{\bf e}_2$ in time $a\ell$ or less, for $\ell=1,2,\ldots,N'$.

Suppose next that the starting point of $\gamma$ lies outside $T_{cN'}(\alpha)$. (We may without loss of generality assume that the starting point is an (internal) boundary point of $\pazocal{C}_\alpha^\theta$.) If the angle $\delta$ exceeds $\pi/4$, then the Euclidean distance from the origin to the exterior of $\pazocal{C}_{\alpha'}^\theta(N')$ is at least $N'/\sqrt{2}$. If $\delta<\pi/4$, then the Euclidean distance from the origin to the exterior of $\pazocal{C}_{\alpha'}^\theta(N')$ is at least
$$
x\sin(\pi/2-\delta)\ge x\sin(\pi/4)\ge \frac{N'\tan(\beta)}{\sqrt{2}}> \frac{N'}{3},
$$
where we have used that $\beta>\pi/6$. A similar calculation shows that a point of $\pazocal{C}_\alpha^\theta$ that lies outside $T_{cN'}(\alpha)$ is at least at Euclidean distance
$$
cN'\frac{\tan(\alpha'-\alpha)}{\cos(\alpha)}
$$
farther to the exterior of $\pazocal{C}_{\alpha'}^\theta(N')$ than the origin. We next recall that the time constant in any direction is at least half its value in the axis direction (due to convexity the shape; it is contained in a square and contains a diamond). It follows that for any $u\in\partial\pazocal{C}_\alpha^\theta\setminus T_{cN'}(\alpha)$ and $v$ in the exterior of $\pazocal{C}_{\alpha'}^\theta(N')$, we have
$$
\mu(v-u)\ge\frac{\mu({\bf e}_1)}{2}|v-u|\ge N'\frac{\mu({\bf e}_1)}{2}\Big[\frac{1}{3}+c\frac{\tan(\alpha'-\alpha)}{\cos(\alpha)}\Big].
$$
On the event $\pazocal{G}\cap\pazocal{G}'$, and since edges outside of the vertical axis have in $\tilde\omega$ only increased, it follows that for such $u$ and $v$ we have, using~\eqref{eq:c_pick}, that
\begin{equation}\label{eq:gamma3}
\tilde T(\gamma)\ge N'\frac{\mu({\bf e}_1)}{12}\Big[1+3c\frac{\tan(\alpha'-\alpha)}{\cos(\alpha)}\Big]>N'(a+\Delta)>N'a+K.
\end{equation}
Again we draw the conclusion that along no such path type~1 may reach $\ell{\bf e}_2$ in time $a\ell$, for $\ell=1,2,\ldots,N'$. This shows that type~2 will reach each such point before type~1, preventing type~1 to pass through them at all. In addition, we note from~\eqref{eq:gamma1} that type~2 will reach $N'{\bf e}_2$ in time $N'a$, while it from~\eqref{eq:gamma1}-\eqref{eq:gamma3} follows that it will take type~1 time exceeding $N'a+K$ to reach any point at the boundary of $\pazocal{C}_{\alpha'}^\theta(N')$. That is, type~2 has its $K$ time units to spare, and will thus reach $N{\bf e}_2$ in time $N'a+K$, at which point type~1 has yet to reach the boundary of $\pazocal{C}_{\alpha'}^\theta(N')$. In conclusion, type~2 will manage to escape type~1 to infinity along a geodesic starting at $N{\bf e}_2$, as desired.

To conclude the proof we note that $\tilde\omega$ is equal to $\omega$ in distribution, and hence that type~2 is able to escape type~1 with positive probability. This ends the proof of the survival part of Theorem~\ref{thm2}.

\section{Non-survival for flat initial conditions}
\label{sec:flat_start}

It remains to show that survival of type~2 is not possible for flat initial conditions. This can be considered as the critical case, balancing between the survival and non-survival regimes. We shall prove the following theorem, extending the result from~\cite{DH07,AP17} to arbitrary directions.

\begin{theorem}\label{thm:flat_start}
Consider competing first-passage percolation on $\Z^2$, with continuous passage time distribution satisfying~\eqref{pt_assumption}, and with flat initial condition $I(\pazocal{C}_{\pi/2}^\theta)$. For every $\theta\in[0,2\pi)$, with probability one, type~2 will infect at most finitely many nodes.
\end{theorem}

We remark that the non-survival part of Theorem~\ref{thm2} is implied by Theorem~\ref{thm:flat_start}. Note also that under the assumptions of Theorem~\ref{thm1}, Theorem~\ref{thmDH} implies that the condition of Theorem~\ref{thm2} is satisfied, so that Theorem~\ref{thm1} is an immediate consequence of Theorems~\ref{thm2} and~\ref{thm:flat_start}. In order to prove Theorem~\ref{thm1} it thus suffices to prove the above theorem.

The most natural approach to prove this theorem would probably be to attempt an adaptation of either of the two existing proofs, in~\cite{DH07,AP17}, treating the case of coordinate directions. We shall instead opt for another approach, which gives us the opportunity to highlight the connection between non-survival for a flat initial condition with the coalescence of infinite geodesics in the geodesic measures constructed in~\cite{DH14}. In fact, we expect that either approach will roughly amount to similar argumentation, as we comment further upon in Remark~\ref{rem:alt_proof} below.

Before attending to the proof of the theorem, we revise the notion of geodesic measures as introduced by Damron and Hanson~\cite{DH14}.

\subsection{Geodesic measures}

Fix a direction $\theta\in[0,2\pi)$, and let $\rho:\R^2\to\R$ be the corresponding linear functional supporting to $\mathcal{A}$ in direction $\theta$. For every $\alpha\in\R$, we let
\begin{equation}\label{eq:ell}
\ell_\alpha^+:=\big\{x\in\R^2:\rho(x)\ge\alpha\big\}\quad\text{and}\quad\ell_\alpha^-:=\big\{x\in\R^2:\rho(x)\le\alpha\big\}
\end{equation}
denote the two half-planes obtained by bisecting the plane in direction $\theta$ at `distance' $\alpha$ from the origin. (Note that `distance' here does not mean Euclidean distance.) For $\alpha\ge0$ we let
$$
\mathcal{F}_\alpha:=\big\{\geo(z,\ell_\alpha^+):z\in\ell_\alpha^-\cap\Z^2\big\}
$$
denote the collection of (finite) geodesics from points in the plane to the half-plane $\ell_\alpha^+$.

Let $\pazocal{E}$ denote the set of nearest-neighbour edges of the square lattice, and let $\bar{\pazocal{E}}$ denote the set of directed nearest-neighbour edges. Let $\Omega_1:=[0,\infty)^\pazocal{E}$, $\Omega_2:=\{0,1\}^{\bar{\pazocal{E}}}$ and $\Omega_3:=(\R^2)^{\Z^2}$. Given a configuration $\omega\in\Omega_1$ of edge weights, and $\alpha\ge0$, let $\eta_\alpha\in\Omega_2$ be the configuration that encodes the edges contained in some geodesic in $\mathcal{F}_\alpha$, with direction towards $\ell_\alpha^+$.  Moreover, let $\zeta_\alpha\in\Omega_3$ be the associated configuration of Busemann differences, defined as
$$
\zeta_\alpha(z):=\big(T(z+{\bf e}_1,\ell_\alpha^+)-T(z,\ell_\alpha^+),T(z+{\bf e}_2,\ell_\alpha^+)-T(z,\ell_\alpha^+)\big).
$$
Consider the map $\Psi_\alpha:\Omega_1\to\Omega_1\times\Omega_2\times\Omega_3$ where $\omega\mapsto(\omega,\eta_\alpha,\zeta_\alpha)$, and let $\nu_\alpha$ denote the push-forward of $\PP$ through the map $\Psi_\alpha$, where $\tilde\Omega:=\Omega_1\times\Omega_2\times\Omega_3$ is equipped with the product topology and the Borel sigma-algebra.

Damron and Hanson proceed with a limiting argument, in order to obtain a measure on infinite geodesics. First, they introduce an averaging step, in order to guarantee a shift invariant measure in the limit. For $n\ge1$, let
$$
\nu_n^\ast(\cdot):=\frac1n\int_0^n\nu_\alpha(\cdot)\,d\alpha.
$$
Observe that, due to subadditivity, $\zeta_\alpha$ is coordinate-wise bounded in absolute value by the edge weights of the corresponding edges. The sequence of measures $(\nu_n^\ast)_{n\ge1}$ is thus tight, and Prokhorov's theorem implies that there exists a weakly convergent subsequence. Damron and Hanson show that any subsequential limit $\nu$ is supported on infinite families of geodesics with nice properties, some of which are summarised in the following theorem.

\begin{theorem}\label{thmDH2}
Consider first-passage percolation on $\mathbb{Z}^2$ with a continuous passage time distribution satisfying~\eqref{pt_assumption}. Let $\theta\in[0,2\pi)$ be a direction and let $\rho:\R^2\to\R$ denote the corresponding supporting linear functional. Every subsequential weak limit $\nu$ of the sequence $(\nu_n^\ast)_{n\ge1}$ is invariant with respect to translations and satisfies the following properties: For $\nu$-almost every $(\omega,\eta,\zeta)\in\tilde\Omega$ we have that
\begin{enumerate}[\quad (a)]
\item for every $z\in\Z^2$ there exists from $z$ a unique forwards path $\Gamma_z$ in $\eta$ which is a geodesic;
\item for every $y,z\in\Z^2$ the geodesics $\Gamma_y$ and $\Gamma_z$ coalesce.
\end{enumerate}
Moreover, if $\rho$ is tangent to $\mathcal{A}$, then $\nu$-almost surely
\begin{enumerate}[\quad (c)]
\item for every $z\in\Z^2$ the Busemann function of the geodesic $\Gamma_z$ is asymptotically linear to $\rho$.
\end{enumerate}
\end{theorem}

\begin{proof}
Parts \emph{(a)}-\emph{(b)} of Theorem~\ref{thmDH2} follow from Theorem~1.11 in~\cite{DH14}. Part~\emph{(c)} is not explicitly stated in~\cite{DH14}, but is a consequence of their Theorem~4.3, Corollary~4.7 and Proposition~5.1.
\end{proof}

We remark that Theorem~\ref{thmDH} is a direct consequence of part~\emph{(c)} of Theorem~\ref{thmDH2}.

\subsection{Preliminary lemmas}

By lattice symmetry, we may without restriction assume that $\theta\in[\pi/4,\pi/2]$. Recall the definition of $\ell_\alpha^-$ in~\eqref{eq:ell}. For $v\in\ell_0^-\cap\Z^2$ we let $C(v)$ denote the set of points infected by $v$ in the competition process when all points in $\ell_0^-\cap\Z^2$ are initially infected, that is,
$$
C(v):=\big\{z\in\Z^2:T(v,z)<T(u,z)\text{ for all }u\in\ell_0^-\cap\Z^2\setminus\{v\}\big\}.
$$
Hence, $G_{\theta,\pi/2}$ coincides with the event $\{|C(0)|=\infty\}$. We note that the laws of the sets $C(v)$ differ for different $v$, in particular if $\theta$ is an irrational direction. This complicates the application of the ergodic theorem, and we shall need to work with a variant of the sets $C(v)$.

Given $v\in\Z^2$, we let $\alpha(v):=\min\{\alpha\in\R:v\in\ell_\alpha^-\}$, and set
$$
C'(v):=\big\{z\in\Z^2:T(v,z)<T(u,z)\text{ for all }u\in\ell_{\alpha(v)}^-\cap\Z^2\setminus\{v\}\big\}.
$$
The distribution of $C'(v)$ is identical for all $v\in\Z^2$, up to translations by the vector $v$. However, for different $v$, the sets $C'(v)$ do not (necessarily) correspond to the same competition problem.

Since we assume that $\theta\in[\pi/4,\pi/2]$, the point ${\bf e}_2=(0,1)$ (and possibly also ${\bf e}_1=(1,0)$) necessarily lies outside $\ell_0^-$. Let $\Delta=\Delta(\theta)$ denote the least $\alpha$ for which this is not the case, i.e.\
$$
\Delta:=\min\big\{\alpha\ge0:{\bf e}_2\in\ell_\alpha^-\big\}=\sup\big\{\alpha\ge0:{\bf e}_2\not\in\ell_\alpha^-\big\}.
$$
We will refer to a point $v\in\ell_\alpha^-\cap\Z^2$ as an (internal) {\bf boundary point} of $\ell_\alpha^-$ if it shares an edge with some point in $\Z^2\setminus\ell_\alpha^-$. In particular, the origin is a boundary point of $\ell_\alpha^-$ if and only if $\alpha\in[0,\Delta)$, and $v\in\Z^2$ is a boundary point in $\ell_\alpha^-$ if and only if $\alpha\in[\alpha(v),\alpha(v)+\Delta)$. Moreover, $\alpha(k{\bf e}_2)=k\Delta$, so that $k{\bf e}_2$ is a boundary point of $\ell_\alpha^-$ if and only if $\alpha\in[k\Delta,(k+1)\Delta)$.

For $v\in\Z^2$, we finally let
$$
C^\ast(v):=\big\{z\in\Z^2:T(v,z)<T(u,z)\text{ for all }u\in\ell_\alpha^-\cap\Z^2\text{ and }\alpha\in[\alpha(v),\alpha(v)+\Delta)\big\}.
$$
Note that $C^\ast(0)$ denotes the set of vertices infected by the origin, not only when all points in $\ell_0^-$ are initially infected, but when all points in $\ell_\alpha^-$ are initially infected for all $\alpha$ for which the origin is a boundary point in $\ell_\alpha^-$. Hence, $C^\ast(0)$ is a subset of $C(0)$, and hence a more restrictive measurement of the progression from the origin in the competition process starting with all points in $\ell_0^-$ being infected.

Note further that the law of $C^\ast(v)$ is identical for all $v\in\Z^2$. Moreover, $C^\ast(v)$ retains information of the competition process with flat initial configuration, as long as we restrict our attention to boundary points of $\ell_0^-$. For this reason we shall work with $C^\ast(v)$, rather than $C(v)$. Considering the events $C^\ast(v)$ will also be important in order to enable a comparison with the Damron-Hanson geodesic measures, at least when $\theta$ is an irrational direction.

We will argue by means of contradiction. Our first goal will be to show that if $\{|C(0)|=\infty\}$ occurs with positive probability, then so does $\{|C^\ast(0)|=\infty\}$. The first step in this is the following lemma.

\begin{lemma}\label{lma:advantage}
Suppose $\PP(|C(0)|=\infty)>0$. Let $v_k$ denote the top-most element in $\ell_0^-\cap\{k\}\times\Z$. Then, for every $M>0$ there exists $N\ge1$ such that
$$
\PP\big(\exists\, g\in\pazocal{T}_0:\,B_g(0,v_k)<0\text{ for } |k|<N,\text{ and }B_g(0,v_k)<-M\text{ for }|k|\ge N\big)>0.
$$
\end{lemma}

\begin{proof}
By Proposition~\ref{prop:geodesic} there exists an event $E$ of full measure such that
\begin{equation}\label{eq:pre_advantage}
\{|C(0)|=\infty\}\cap E\subseteq\big\{\exists\, g\in\pazocal{T}_0:\,T(0,z)<T(v_k,z)\text{ for every }z\in g\text{ and }k\neq0\big\}.
\end{equation}
We note, in particular, that on the above event we have $B_g(0,v_k)\le0$ for all $k\in\Z$.

Fix $M>0$. Let $r\ge0$ denote the infimum of the support of the edge-weight distribution. Fix $\delta>0$ and let $A_n$ denote the event that the average weight per edge, along every path from the origin to some points at $\ell_1$-distance $n$, is at least $r+2\delta$. A standard path counting argument shows that for $\delta>0$ small enough, the probability of $A_n$ tends to 1 as $n\to\infty$. We fix $n\ge M/\delta$ so that $\{|C(0)|=\infty\}\cap E\cap A_n$ occurs with positive probability.

For any self-avoiding path $\gamma$ between the origin and some point at $\ell_1$-distance $n$, let $C_\gamma$ denote the event that some geodesic $g$ satisfying the property in~\eqref{eq:pre_advantage} contains $\gamma$. Since $\{|C(0)|=\infty\}\cap E\cap A_n$ occurs with positive probability, also $C_\gamma$ occurs with positive probability for some path $\gamma$. Let $C_{\gamma,N}$ denote the event that $C_\gamma$ occurs, and that $T(v_k,\gamma)>M$ for all $|k|\ge N$. By the shape theorem, for large enough $N$ also $C_{\gamma,N}$ occurs with positive probability. Finally, we choose $\vep\in(0,\delta]$ so that
$$
\PP\big(C_{\gamma,N}\cap\{\omega_e>r+2\vep\text{ for all }e\in\gamma\}\big)>0.
$$

Let $\omega$ and $\omega'$ be two independent weight configurations, and construct $\tilde\omega$ by identifying $\tilde\omega$ with $\omega'$ on $\gamma$, and with $\omega$ elsewhere. Let $D_\gamma$ denote the event that $\{\omega'_e\le r+\vep\text{ for all }e\in\gamma\}$. Again $D_\gamma$ occurs with positive probability, and since $\omega$ and $\omega'$ are sampled independently, it follows that
$$
\PP\big(C_{\gamma,N}\cap\{\omega_e>r+2\vep\text{ for all }e\in\gamma\}\cap D_\gamma\big)=\PP\big(C_{\gamma,N}\cap\{\omega_e>r+2\vep\text{ for all }e\in\gamma\}\big)\PP(D_\gamma)>0.
$$

Now suppose that $C_{\gamma,N}\cap\{\omega_e>r+2\vep\text{ for all }e\in\gamma\}\cap D_\gamma$ occurs. Since $g$ is a geodesic in $\omega$ which contains $\gamma$, and since the configurations $\tilde\omega$ and $\omega$ differ only on $\gamma$, so that on the event $D_\gamma$ the configuration $\tilde\omega$ is coordinate-wise smaller than $\omega$, it follows that $g$ is a geodesic also in $\tilde\omega$. That $B_g(0,v_k)\le0$ for all $k\in\Z$ remains true also in $\tilde\omega$. Moreover, travelling from the origin along $g$ to distance $n$ has in $\tilde\omega$ become (at least) $\delta n\ge M$ less expensive. As a consequence, for every $|k|\ge N$, for which the travel time from $v_k$ to any vertex along $\gamma$ is at least $M$ we will, in $\tilde\omega$, have $B_g(0,v_k)<-M$.

The proof ends with us noting that $\omega$ and $\tilde\omega$ are equal in distribution, and hence that the conclusion of the lemma has been proven.
\end{proof}

The above lemma introduced a notation for boundary points of $\ell_0^-$, which correspond to the set of initially infected sites in the competition process. Since we shall need to consider translates of this half-plane, we extend this notation to boundary points of $\ell_\alpha^-$ for all $\alpha\in\R$.

Given $\alpha\in\R$ and $k\in\Z$, let $v_k^\alpha$ denote the top-most element in $\ell_\alpha^-\cap\{k\}\times\Z$. Since we assume that $\theta\in[\pi/4,\pi/2]$, the point $v_k^\alpha$ is well-defined, and corresponds to the unique (internal) boundary point of $\ell_\alpha^-$ along the vertical line $\{k\}\times\Z$. That is, the bi-infinite sequence $(v_k^\alpha)_{k\in\Z}$ is an enumeration of all boundary points of $\ell_\alpha^-$.

\begin{lemma}\label{lma:Cast}
If $\PP(|C(0)|=\infty)>0$, then also $\PP(|C^\ast(0)|=\infty)>0$.
\end{lemma}

\begin{proof}
Suppose $\theta\in[\pi/4,\pi/2]$, and that $\PP(|C(0)|=\infty)>0$. Let $r$ denote the infimum of the support of the weight distribution, and fix $\delta>0$ so that $\PP(\omega_e\ge r+2\delta)>0$. Next, pick an integer $m>1+(r+\delta)/\delta$.

For $N\ge1$ let $A_{m,N}$ denote the event that there exists $g\in\pazocal{T}_{m{\bf e}_2}$ such that
$$
B_g(m{\bf e}_2,v_k^{m\Delta})<0\text{ for } |k|<N,\text{ and }B_g(m{\bf e}_2,v_k^{m\Delta})<-m(r+\delta)\text{ for }|k|\ge N.
$$
Moreover, let
$$
A'_{m,N}:=\big\{T(v_k^{m\Delta},\{0\}\times\Z)>m(r+\delta)\text{ for all }|k|\ge N\big\},
$$
and for $M\ge0$, let
$$
A''_{m,N}:=\big\{\omega_e\le M\text{ for all $e$ in }(\ell_{m\Delta}^-\setminus\ell_0^-)\cap[-2N,2N]\times\R\big\}.
$$
By Lemma~\ref{lma:advantage} there exists $N\ge1$ such that $\PP(A_{m,N})>0$, and this remains true for all larger values of $N$ since $A_{m,N}$ is increasing in $N$. In addition, by the shape theorem, the probability that $A'_{m,N}$ occurs can be made arbitrarily close to 1 by increasing $N$. Finally, due to continuous passage times, we may make the probability of $A''_{m,N}$ occurring arbitrarily close to 1 by taking $M=M(N)$ large, while $\PP(\omega_e>M)>0$ remains true. That is, there exist $N\ge m$ and $M>r+2\delta$ such that both $\PP(\omega_e>M)>0$ and 
$$
\PP\big(A_{m,N}\cap A'_{m,N}\cap A''_{m,N}\big)>0.
$$
We fix $N\ge m$ and $M>r+2\delta$ accordingly.

We proceed once again with a resampling argument. Sample $\omega'$ independently of $\omega$ according to the same law, and let $\tilde\omega$ be the configuration which on $(\ell_{m\Delta}^-\setminus\ell_0^-)\cap[-2N,2N]\times\R$ coincides with $\omega'$, and which coincides with $\omega$ elsewhere. Hence, $\tilde\omega$ again has the law of $\omega$. Let
\begin{align*}
D_{m,N}:=&\big\{\omega'_e\le r+\delta\text{ for $e=((k-1){\bf e}_2,k{\bf e}_2)$ and }k=1,2,\ldots,m\big\}\\
&\cap\big\{\omega'_e\ge M\text{ for remaining $e$ in }(\ell_{m\Delta}^-\setminus\ell_0^-)\cap[-2N,2N]\times\R\big\}.
\end{align*}
By the choice of $M$, the event $D_{m,N}$ occurs with positive probability, and since $\omega$ and $\omega'$ are independent, we have
\begin{equation}\label{eq:pos_proba}
\PP\big(A_{m,N}\cap A'_{m,N}\cap A''_{m,N}\cap D_{m,N}\big)>0.
\end{equation}

We pause for a geometric observation. When $\theta=\pi/2$, then $\ell_0^-$ coincides with the lower half-plane, and the boundary points of $\ell_0^-$ all lie on the horizontal axis. For every $j\ge1$ the $\ell_1$-distance from any boundary point to $j{\bf e}_2$ is at least $j+1$. As $\theta$ decreases from $\pi/2$ to $\pi/4$, the distance to $j{\bf e}_2$ decreases as new boundary points appear. The distance is at its minimum for $\theta=\pi/4$, in which case there for every $j$ exists a boundary point of $\ell_0^-$ at distance $j$ from the point $j{\bf e}_2$. Note that for $\theta=\pi/4$, the set of boundary points coincide for $\ell_\alpha^-$ for all $\alpha\in[0,\Delta)$. 

Now, on the event that $A_{m,N}\cap A'_{m,N}\cap A''_{m,N}\cap D_{m,N}$ occurs, the following is true in the weight configuration $\tilde\omega$:
\begin{enumerate}[\quad (i)]
\item $T(0,j{\bf e}_2)\le(r+\delta)j$ for all $j=1,2,\ldots,m$;
\item $T(v_k^{m\Delta},j{\bf e}_2)>m(r+\delta)$ for all $j=1,2,\ldots,m$ and $|k|\ge N$;
\item $T(v_k^\Delta,j{\bf e}_2)>(r+2\delta)j$ for all $j=1,2,\ldots,m$ and $0<|k|<N$;
\item for any infinite path $g$ starting at $m{\bf e}_2$ satisfying the condition of $A_{m,N}$ (and hence is a geodesic in $\omega$) we have for every $z\in g$ that $T(m{\bf e}_2,z)<T(v_k^{m\Delta},z)$ for all $k\neq0$ and that $T(0,z)<T(v_k^{m\Delta},z)$ for all $|k|\ge N$.
\end{enumerate}
We now claim two things, still on the event $A_{m,N}\cap A'_{m,N}\cap A''_{m,N}\cap D_{m,N}$, and for any path $g$ as in~(iv): First that $T(v_k^{\Delta},j{\bf e}_2)-T(0,j{\bf e}_2)>0$ for all $j=1,2,\ldots,m$ and $k \neq 0$, and second that for every $z\in g$ we have $T(v_k^\Delta,z)-T(0,z)>0$ for all $k\neq 0$. Since any other path crossing at $k=0$ is also strictly slower than the origin due to the low weights on the vertical axis by (i), these claims imply that $z\in C^\ast(0)$ for every $z\in g$, and hence that $|C^\ast(0)|=\infty$, which by~\eqref{eq:pos_proba} would complete the proof.

We first note that by~(i) and~(iii) we have $T(v_k^\Delta,j{\bf e}_2)-T(0,j{\bf e}_2)>0$ for all $j=1,2,\ldots,m$ and $0<|k|<N$, whereas~(i) and~(ii) give the same for $|k|\ge N$. This verifies the first claim.

We next fix $z\in g$ and $k \neq 0$, and let $\gamma$ be any path from $v_k^\Delta$ to $z$. To verify the second claim it suffices to argue that $T(\gamma)-T(0,z)>0$. We first note that if $\gamma$ visits $j{\bf e}_2$ for some $j=1,2,\ldots,m$, then $T(\gamma)-T(0,z)>0$ by the first claim. It will therefore suffice to consider $\gamma$ that does not visit vertices of this type. To reach $z$, however, $\gamma$ must visit $v_i^{m\Delta}$ for some $i\neq0$. Let $\gamma_-$ denote the segment of $\gamma$ from $v_k^\Delta$ to the first such visit, and $\gamma_+$ denote the remainder of $\gamma$. If the first visit is to $v_i^{m\Delta}$ for some $|i|<N$, then $\gamma_-$ must pass at least $m$ edges of weight at least $M>r+2\delta$, and so by~(i) and~(iv) we have
$$
T(\gamma)-T(0,z)>T(\gamma_-)-T(0,m{\bf e}_2)+T(v_i^{m\Delta},z)-T(m{\bf e}_2,z)>m\delta+0>0.
$$
If instead the first visit is to $v_i^{m\Delta}$ for some $|i|\ge N$, then by~(iv)
$$
T(\gamma)-T(0,z)\ge T(v_i^{m\Delta},z)-T(0,z)>0.
$$
This verifies the second claim, and hence completes the proof.
\end{proof}

\subsection{Proof of Theorem~\ref{thm:flat_start}}

We are now ready to proceed to the proof of Theorem~\ref{thm:flat_start}. We shall argue by contradiction, and show that survival with positive probability is incompatible with the Damron-Hanson geodesic measures being supported on coalescing families of geodesics, contradicting Theorem~\ref{thmDH2}.

We may assume that $\theta\in[\pi/4,\pi/2]$, due to lattice symmetry. Aiming for a contradiction, we assume henceforth that $\PP(|C(0)|=\infty)>0$, so that by Lemma~\ref{lma:Cast} also $\PP(|C^\ast(0)|=\infty)>0$. 
We next claim that for every $\alpha\in\R$, almost surely,
\begin{equation}\label{eq:Malpha}
\lim_{n\to\infty}\frac{1}{n}\sum_{k=1}^n\mathbf{1}_{\{|C^\ast(v_k^\alpha)|=\infty\}}=\lim_{n\to\infty}\frac{1}{n}\sum_{k=-n}^{-1}\mathbf{1}_{\{|C^\ast(v_k^\alpha)|=\infty\}}=\PP\big(|C^\ast(0)|=\infty\big)>0.
\end{equation}
When $\theta$ is a rational direction, i.e.\ corresponding to $z/|z|$ for some $z\in\Z^2$, this follows from Birkhoff's standard point-wise ergodic theorem. That this remains true in irrational directions $\theta$ follows from the ergodic theorems covering directional subsequences; see e.g.~\cite[Chapter~8]{K85}.

The statement in~\eqref{eq:Malpha} holds for fixed $\alpha\in\R$, and the almost sure limit does not depend on $\alpha$. For $\alpha=0$, the statement can be interpreted in terms of the competition process with flat initial condition:
%If $\PP(|C(0)|=\infty)>0$, then by Lemma~\ref{lma:Cast} also $\PP(|C^\ast(0)|=\infty)>0$, and 
With probability one, a positive density of points along the boundary of $\ell_0^-$ will each be responsible for the infection of infinitely many vertices. The next step will be to translate this statement into a density statement for points in $\ell_0^+$ reached by different $v\in\ell_\alpha^-$, where $\alpha<0$. In other words, for $\alpha<0$ being fixed, vertices in $\ell_0^+$ belonging to distinct sets $C^\ast(v_k^\alpha)$ for $k\in\Z$.

Fix $\alpha<0$ and an even integer $m\ge 1+2/\vep$, where $\vep:=\PP(|C^\ast(0)|=\infty)>0$. For $k\in\Z$ let $I_{m,k}$ denote the set containing the $m$ boundary points of $\ell_0^+$, i.e.\ points of the form $v_k^0$, contained in the vertical strip $km{\bf e}_1+\{-m/2+1,\ldots,m/2\}\times\Z$. Together, the sets $\{I_{m,k}:k\in\Z\}$ form a partition of the (internal) boundary points of $\ell_0^+$. Recall that $\alpha<0$, and let $X^\alpha_{m,k}$ denote the number of boundary vertices $v_k^\alpha$ of $\ell_\alpha^-$ for which $C^\ast(v_k^\alpha)\cap I_{m,k}\neq\emptyset$. By the ergodic theorem along directional subsequences, we have for every fixed $\alpha<0$ that, almost surely,
\begin{equation}\label{eq:Xalpha}
\lim_{n\to\infty}\frac{1}{2n+1}\sum_{k=-n}^nX^\alpha_{m,k}=\E[X^\alpha_{m,0}].
\end{equation}

Let $\Lambda$ denote the set of $\alpha\in(-\infty,0)$ for which both~\eqref{eq:Malpha} and~\eqref{eq:Xalpha} occur. Let $\pazocal{L}$ denote Lebesgue measure. Then, by Fubini's theorem, for every $N<\infty$ we have
$$
\E\big[\pazocal{L}\big(\Lambda\cap[-N,0]\big)\big]=\int_{[-N,0]}\PP\big(\text{\eqref{eq:Malpha} and \eqref{eq:Xalpha} hold for $\alpha$}\big)\,d\alpha=N,
$$
and hence that $\Lambda$ has full Lebesgue measure as a subset of $(-\infty,0)$, almost surely. Consider $\alpha\in\Lambda$. By~\eqref{eq:Malpha} we have for all large $n$ that
$$
\sum_{k=(m-1)n+1}^{mn}\mathbf{1}_{\{|C^\ast(v_k^\alpha)|=\infty\}}\ge\vep n/2\quad\text{and}\quad\sum_{k=-[(m-1)n+1]}^{-mn}\mathbf{1}_{\{|C^\ast(v_k^\alpha)|=\infty\}}\ge\vep n/2.
$$
Pick $n\gg|\alpha|/\vep$ such that the above holds. Since $\vep n$ is large in comparison to $|\alpha|$, there is not enough room for the disjoint paths emanating from the $v_k^\alpha$, for $|k|\in((m-1)n,mn]$, whose $C^\ast$-sets are infinite, to escape through the vertical segment of $\{\pm mn\}\times\R$ contained in $\ell_0^-\setminus\ell_\alpha^-$.
This means that for all $|k|\le(m-1)n$ for which $|C^\ast(v_k^\alpha)|=\infty$, the sets $C^\ast(v_k^\alpha)$ will each necessarily have to contain a boundary point of $\ell_0^+$ contained in the strip $[-mn,mn]\times\R$. That is, for $\alpha\in\Lambda$ we have for all large $n$ that
$$
\frac{1}{2n+1}\sum_{k=-n}^nX_{m,k}^\alpha\ge\frac{1}{2n+1}\sum_{k=-(m-1)n}^{(m-1)n}\mathbf{1}_{\{|C^\ast(v_k^\alpha)|=\infty\}}.
$$
Sending $n\to\infty$, 
%For each fixed $\alpha>0$,
it follows from~\eqref{eq:Malpha} and~\eqref{eq:Xalpha} that
$$
\E[X_{m,0}^\alpha]\ge (m-1)\,\PP(|C^\ast(0)|=\infty)=(m-1)\vep.
$$
Together with the trivial bounds that $0\le X_{m,0}^\alpha\le m$, it follows that
$$
(m-1)\vep\le\E[X_{m,0}^\alpha]\le\PP(X_{m,0}^\alpha\le1)+m\,\PP(X_{m,0}^\alpha\ge2)=1+(m-1)\,\PP(X_{m,0}^\alpha\ge2).
$$
Rearranging the terms in the above equation, and recalling that $m\ge1+2/\vep$, it follows that for every $\alpha\in\Lambda$ we have that
\begin{equation}\label{eq:Xalpha2}
\PP(X_{m,0}^\alpha\ge2)\ge \vep-\frac{1}{m-1}\ge\vep/2.
\end{equation}
Now, the set $\Lambda$ is a priori random, but the quantity $\PP(X_{m,0}^\alpha\ge2)$ is not. The set $\Lambda$ has nevertheless full Lebesgue measure, and we conclude that the (nonrandom) set of $\alpha\in(-\infty,0)$ for which~\eqref{eq:Xalpha2} holds has full Lebesgue measure.

To finish the proof, we shall argue that the conclusion in~\eqref{eq:Xalpha2} implies that the Damron-Hanson geodesic measure $\nu$ would have to put positive mass on non-coalescing families of geodesics, contradicting Theorem~\ref{thmDH2}.

Set $\theta'=\theta+\pi$, and let $(\nu_\alpha)_{\alpha\ge0}$ denote the sequence of measures obtained as the push-forward through the map $\Psi_\alpha$ with respect to the direction $\theta'$. That is, $\nu_\alpha$ is a measure supported on families of geodesics between points in $\Z^2$ and the line $\ell_{-\alpha}^-$. Let $m$ be as above. For $k\ge1$ let $A_k$ denote the event that there exist $x,y\in I_{m,0}$ such that the paths encoded in $\eta$ starting at $x$ and $y$ remain disjoint until they first reach the boundary of the box $[-k,k]^2$. Then,
$$
\nu_\alpha(A_k)\ge\PP(X_{m,0}^\alpha\ge2),
$$
and since the set of $\alpha$ for which~\eqref{eq:Xalpha2} holds has full Lebesgue measure, we have
$$
\nu_n^\ast(A_k)=\frac1n\int_0^n\nu_\alpha(A_k)\,d\alpha\ge\vep/2.
$$
Since $A_k$ is an event depending on finitely many edges, it is closed. Hence, by the portmanteau theorem, it follows that for any subsequential limit $\nu$ of the sequence $(\nu_n^\ast)_{n\ge1}$ we have
$$
\nu(A_k)\ge\limsup_{n\to\infty}\nu_n^\ast(A_k)\ge\vep/2,
$$
uniformly in $k\ge1$. Sending $k\to\infty$, since $A_{k+1}\subseteq A_k$, continuity of measure gives that
$$
\nu\Big(\bigcap_{k\ge1}A_k\Big)=\lim_{k\to\infty}\nu(A_k)\ge\vep/2.
$$
On the event $\bigcap_{k\ge1}A_k$ there are vertices $x,y\in I_{m,0}$ such that the infinite paths encoded by $\eta$ and starting at $x$ and $y$ are disjoint. Hence, we conclude that with probability at least $\vep/2>0$ the measure $\nu$ is supported on families of geodesics that do not all coalesce. This is contradicted by part~\emph{(b)} of Theorem~\ref{thmDH2}.

\begin{remark}\label{rem:alt_proof}
It is conceivable that an alternative proof of Theorem~\ref{thm:flat_start} is possible, not relying on the work of~\cite{DH14}. One such approach would be to first note that an argument analogous to that giving~\eqref{eq:Xalpha2} will also give the stronger bound $\PP(X_{m,0}^\alpha\ge3)\ge\vep/2$, assuming $m\ge2+4/\vep$. One could then attempt to construct a blocking event, having a positive probability to occur, such that if it occurs, only two of the clusters counted by $X_{m,0}^\alpha$ could ``survive'' for an additional fixed number of steps. Since these blocking events ought to occur eventually, by virtue of being local and occurring with positive probability, they would decrease the number of surviving infections as $\alpha\to-\infty$. This ought to result in a contradiction with the initial assumption that $\PP(|C(0)|=\infty)>0$.

We expect that this alternative proof strategy, if viable, would amount to roughly the same calculations as the proof given above. We think that the above proof has its affordances, in that it illustrates the connection between coalescence and coexistence. However, obtaining a proof independent of~\cite{DH14} would also be of value, as it could be used to simplify some rather technical computations in showing that the geodesic measures introduced in that work are supported on coalescing families of geodesics.
\end{remark}

\paragraph{Acknowledgments.} This work was supported by the Swedish Research Council through the grants 2020-04479\_VR (MD and MS) and 2021-03964\_VR (DA).

\newcommand{\noopsort}[1]{}\def\cprime{$'$}

\end{document}